\documentclass[final,12pt,3p]{elsarticle}

\usepackage[english]{babel}

\usepackage{amsmath}
\usepackage{amssymb}
\usepackage{graphicx}
\usepackage[colorlinks=false]{hyperref}
\usepackage{bm}
\usepackage{mathtools}
\usepackage{hhline}
\usepackage{amsthm}
\usepackage{subcaption}
\usepackage[dvipsnames]{xcolor}
\usepackage[normalem]{ulem}
\usepackage{multirow}
\usepackage{stmaryrd}

\usepackage{listings}

\usepackage{algpseudocode}
\usepackage{algorithm}
\usepackage{tikz}
\usetikzlibrary{shapes.geometric, arrows.meta, calc}
\algnewcommand{\IIf}[1]{\State\algorithmicif\ #1\ \algorithmicthen}
\algnewcommand{\EndIIf}{\unskip\ \algorithmicend\ \algorithmicif}

\theoremstyle{remark} 
\newtheorem{remark}{Remark}

\theoremstyle{plain} 
\newtheorem{lemma}{Lemma}

\theoremstyle{plain} 
\newtheorem{theorem}{Theorem}

\theoremstyle{definition}
\newtheorem{definition}{Definition}

\usepackage{setspace}
\AtBeginEnvironment{tabular}{\singlespacing}
\AtBeginEnvironment{algorithm}{\singlespacing}

\journal{Journal of Computational Physics}

\begin{document}

\begin{frontmatter}

\title{A Space-time Approach to Entropy-Stable Discontinuous Galerkin and Flux Reconstruction}

\author[mcgill]{Carolyn M. V. Pethrick}
\ead{carolyn.pethrick@mail.mcgill.ca}

\author[mcgill]{Siva Nadarajah}
\ead{siva.nadarajah@mcgill.ca}

\affiliation[mcgill]{organization={Department of Mechanical Engineering, McGill University},
            addressline={845 Sherbrooke St W}, 
            city={Montreal},
            postcode={H3A 0G4}, 
            state={Quebec},
            country={Canada}}

\begin{abstract}

We present a high-order space-time discretization equipped with fully-discrete entropy stability properties
for general choices of volume and surface quadrature rules.
The formulation uses flux reconstruction (FR) in the spatial dimension paired with a discontinuous Galerkin (DG) method in the temporal dimension. The result is a fully-implicit system using polynomial bases in space and time.
An energy-stable discretization is applied to the linear advection equation, yielding optimal $p+1$ convergence for small FR correction parameters and $p$ convergence at the same filter strength as method-of-lines implementations. We can thus recover the space-time equivalent to schemes such as DG, Huynh's FR, or spectral difference through a single parameter $c$.
We follow with a similar space-time nonlinearly-stable flux reconstruction (ST-NSFR) scheme, which uses skew-symmetric stiffness operators in both space and time. The ST-NSFR scheme is fully-discretely entropy preserving using the $c_{DG}$ parameter or entropy-stable for small $c$.
Numerical experiments using the linear advection and Euler equations confirm convergence orders and stability properties. The advantage of FR in a space-time context 
is demonstrated by a reduction in computational cost up to about $70\%$ as $c$ is increased. 
\end{abstract}


\end{frontmatter}

\section{Introduction}
\label{introduction}

High-order methods equipped with entropy stability properties have garnered attention for their robustness and predictive capacity. While there are numerous well-developed options for entropy-stable and high-order spatial methods, the choice of temporal integration scheme remains an open topic.

In this work, we are particularly interested in flux reconstruction (FR) approaches. By applying a correction to the discontinuous flux, FR methods have the favourable property of increasing the explicit time step size limit~\cite{huynh_flux_2007}. 
In addition to increasing the stable time step size, locally applying FR has been shown to reduce spurious oscillations near discontinuities~\cite{srinivasan2025adaptivefluxreconstructionscheme}.

Entropy-stable methods have addressed the robustness challenge arising in HO methods. Applying principles first developed by Harten~\cite{harten1983self,harten1983symmetric,harten1987uniformly}, entropy stability results in a guarantee of stability for the nonlinear PDEs which represent many physically-relevant systems. In the past decade, HO researchers have built on the work of Fisher \textit{et al.}~\cite{fisher2013discretely} and Gassner~\cite{gassner2013skew} to formulate entropy-stable HO methods. In this work, we use the nonlinearly-stable FR (NSFR) method of Cicchino, Nadarajah and co-authors~\cite{cicchino2022nonlinearly,cicchino2022provably,cicchino2025discretely} as a spatial discretization, which has a semi-discrete entropy stability property.

Semidiscrete HO formulations are typically paired with ODE methods in a method-of-lines (MOL) approach. 
Explicit RK is the most standard choice for its ease of implementation. 
However, the Courant-Friedrichs-Lewy (CFL) condition causes the maximum explicit time step size to have inverse proportionality to the polynomial degree in HO methods~\cite{cockburn2001Runge}.
To circumvent the CFL condition, implicit MOL approaches can be used. Several authors explore options for implicit time stepping for HO methods in the context of high Reynolds number aerodynamic flows~\cite{hartmann2021implicit}. 
Implicit or implicit-explicit (IMEX) MOL temporal approaches have been applied to FR semidiscretizations, including~\cite{vermeire2016implicit,jia2019accuracy,wang2019implicit, vandenhoeck2019implicit,pereira2024hybridized}. In~\cite{pereira2024hybridized}, the cost of the implicit solution for an FR method was found to be lower than the equivalent DG method when using a diagonally-implicit RK temporal integration method.

Rather than the typical pairing of a spatial semidiscretization and a method-of-lines temporal discretization, space-time methods discretize space and time together. Here, we consider space-time methods to be those which solve spatial and temporal variables together, such that the spatial and temporal parts are not easily separable as in the method-of-lines.
In the context of FR, three types of space-time discretizations have been developed.
First,  Lu \textit{et al.}~\cite{lu2015flow} and Zhang \textit{et al.}~\cite{zhang2017high} apply the space-time expansion of Gassner~\cite{lorcher2007discontinuous,gassner2008discontinuous} to use the maximum time step size to be used in each element, achieving efficiency improvements up to $3.35\times$.
Second, the compact Runge-Kutta FR scheme proposed by Babbar and Chen~\cite{babbar2025compact} applies a FR correction only to the final stage of an RK method, coupling the spatial and temporal variables. Doing so, their method is completely local during stage calculations, exchanging numerical flux only at the end of a time step~\cite{babbar2025compact}.
Finally, the approach of ~\cite{yu2017nodal, mccaughtry2021high, yu2025highordernodalspacetimeflux} simultaneously discretizes space and time using FR correction functions in a $dim+1$ system. Yu~\cite{yu2017nodal} solves advection-dominated problems, while McCaughtry \textit{et al.}~\cite{mccaughtry2021high} advance to advection-diffusion problems. Using correction functions recovering DG, Yu~\cite{yu2025highordernodalspacetimeflux} demonstrates temporal superconvergence on curvilinear grids.
The work herein is distinct from the categories existing in literature due to the implementation of FR as DG with a modified mass matrix and the application of FR only in space.

If fully-discrete entropy stability is desired in MOL schemes, the relaxation Runge-Kutta method of~\cite{ketcheson2019relaxation,ranocha2020relaxation} can be used.
In this work, we explore the space-time alternative to entropy stability. We take inspiration from the scheme of Friedrich \textit{et al.}~\cite{friedrich2019entropy}, which uses skew-symmetric operators for both the spatial and temporal discretizations in a DGSEM space-time discretization. They develop entropy stability and preservation properties for the resulting scheme. Herein, we develop a space-time scheme with the same philosophy of split forms in both space and time. Our nonlinearly-stable approach, based on the semidicretization of ~\cite{cicchino2022nonlinearly,cicchino2022provably,cicchino2025discretely}, adds the novelty of flux reconstruction and uncollocated nodes. Furthermore, we improve on the fully-discrete stability property of previous MOL work~\cite{pethrick2025fully}, where the RRK approach was only able to address the $L_2$ part of entropy change for an FR discretization. The present work demonstrates stability when the FR $c$ parameter{\color{black}, which is used to vary the strength of FR, }is small. 

This work formulates a space-time energy stable FR (ST-ESFR) and a corresponding space-time nonlinearly stable FR (ST-NSFR) scheme. Section~\ref{sec:prelim} introduces preliminaries for the new schemes. By applying flux reconstruction only in the spatial dimension, we verify that the ST-ESFR method described in Section~\ref{sec:ST-ESFR} recovers well-known properties of MOL ESFR methods. The corresponding ST-NSFR scheme, which applies split forms in both space and time, is introduced in Section~\ref{sec:ST-NSFR}. Therein, we prove discrete conservation, entropy preservation and entropy stability for the $c_{DG}$ scheme. A key result of this work is the identification of a range of the FR parameter $c$ for which the broken Sobolev inner product remains non-negative. We show that, for such choices of $c$, the contribution from higher-order derivatives is sufficiently controlled such that the Sobolev inner product is bounded below by a non-negative quantity{\color{black}, a property which implies entropy stability for small $c$.} Section~\ref{sec:Burgers-results} demonstrates the implications of our analysis on a scalar-valued conservation law and presents numerical demonstrations of entropy stability and preservation for Burgers' equation. We follow with a numerical demonstration of the properties of ST-NSFR for vector-valued conservation laws in Section~\ref{sec:results-Euler}, where we demonstrate entropy stability for a range of $c$ values. {We observe stability for FR schemes, including Huynh's FR~\cite{huynh_flux_2007}, emphasizing the significance of the stability property.}
Finally, in Section~\ref{sec:comp_cost}, we demonstrate that the application of flux reconstruction consistently reduces the cost required to solve the implicit system for reasonably small correction parameters.

\section{Preliminaries and Computational Domain}
\label{sec:prelim}

\subsection{Notation}

We use the typical notation that scalars use italic symbols $a$, vectors use boldface $\bm{a}$, and matrices use capital or calligraphic $\mathcal{A}$ or $\bm{A}$. We note that we understand vectors as row vectors, i.e., $\bm{a} = [a_0 \ a_1\ ... a_n] \in \mathbb{R}^n$, consistent with previous work, cf.~\cite{zwanenburg2016equivalence, cicchino2021new,cicchino2025discretely}.

\subsection{Entropy stability analysis for a continuous PDE}
\label{sec:entropy-stability-CTS}

We assume $1D+1$ domain, with one spatial dimension and time discretized as an additional dimension. Consider a general vector-valued PDE  on the 1D+1 domain, periodic in the spatial dimension:
\begin{equation}
    \begin{cases}
        \frac{\partial}{\partial t}\bm{f}_t(\bm{u}({x},t)) +\frac{\partial}{\partial x}\bm{f}_s(\bm{u}({x},t) ) = 0,\quad \text{in } [ x_L,x_R ] \times [0,T]\\
        \bm{u}({x},0)=\bm{u}_0(x) \\
        \bm{u}(x_L,t) = \bm{u}(x_R,t),
    \end{cases}
    \label{eq:general_cons_law}
\end{equation}
where $\bm{f}_t=\bm{u}$ is the state.
The analyses in this work are easily extended to multiple spatial dimensions, but we focus on the $1D+1$ case for notational brevity. We direct the reader to the work of Friedrich \textit{et al.}~\cite{friedrich2019entropy} for a formulation considering the multi-dimensional, curvilinear case.
We consider that the PDE of Eq.~\eqref{eq:general_cons_law} admits a convex entropy function $s = s(\bm{u})$ with associated entropy flux $\bm{F}$, a set of entropy variables
\begin{equation}
    \bm{v} = \frac{\partial s(\bm{u})}{\partial \bm{u}},
    \label{eq:entropy_vars}
\end{equation}
and the entropy flux-entropy potential pair
\begin{equation}
    (\phi, \bm{\psi}) = \left( \bm{v}^T \bm{u} - s, \bm{w}^T \bm{f} - \bm{F}
    \right).
\end{equation}
{
Then left multiplying the conservation law yields a statement of entropy conservation for smooth solutions,
\begin{equation}
    \frac{\partial s(\bm{u})}{\partial t} + \frac{\partial \bm{F}(\bm{u})}{\partial x} = 0.
    \label{eq:entropy_conservation_in_diffform}
\end{equation}
Integrating~(\ref{eq:entropy_conservation_in_diffform}) over the space-time domain $[-1,1] \times [0,T]$ and using the definition of the entropy potential yields
\begin{equation}
\begin{aligned}    
&\int_{-1}^1 \left[ \bm{v}^T\bm{u} -\phi(\bm{v}) \right]_0^T \;\mathrm{d}\Omega_s + \int_0^T  \left[ \bm{v}^T\bm{f}_s(\bm{u}(\bm{v})) -\psi(\bm{v}) \right]_{-1}^1\;\mathrm{d}t = 0.\\
&\int_{-1}^1 s(\bm{u}(T)) - s(\bm{u}(0))\;\mathrm{d}\Omega_s + \int_0^T  \left[ \bm{v}^T\bm{f}_s(\bm{u}(\bm{v})) -\psi(\bm{v}) \right]_{-1}^1\;\mathrm{d}t = 0.
\label{eq:entropy_conservation_in_intform1}
\end{aligned}
\end{equation}
In a more general setting, physically relevant solutions of~\eqref{eq:general_cons_law} satisfy the entropy inequality,
\begin{equation}
\int_{-1}^1 s(\bm{u}(T))  \;\mathrm{d}\Omega_s \le
\int_{-1}^1 s(\bm{u}(0)) \;\mathrm{d}\Omega_s
- \int_0^T  \left[ \bm{v}^T\bm{f}_s(\bm{u}(\bm{v})) -\psi(\bm{v}) \right]_{-1}^1\;\mathrm{d}t.
\label{eq:entropy_conservation_in_intform2}
\end{equation}
}
{Upon the introduction of a finite-difference discretization identified by superscripts $(\cdot)^h$, Friedrich~\textit{et al.}~\cite{friedrich2019entropy} derive a discrete version of the temporal entropy condition 
\begin{equation}
        \llbracket \bm{v}^h\rrbracket \bm{f}_t (\bm{u}^h_i, \bm{u}^h_j) = \llbracket\phi\rrbracket 
    \label{eq:Tadmor-shuffle},
\end{equation}
where $\llbracket\cdot\rrbracket$ is a volume jump; two-point temporal state functions can be found similar to existing two-point spatial fluxes~\cite{friedrich2019entropy}. The entropy condition~\eqref{eq:Tadmor-shuffle} is the temporal counterpart of the spatial entropy condition
\begin{equation}
        \llbracket\bm{v}^h\rrbracket \bm{f}_s (\bm{u}^h_i, \bm{u}^h_j) = \llbracket\bm{\psi}\rrbracket.
    \label{eq:Tadmor-shuffle-space}
\end{equation}}
Friedrich\textit{ et al.}~\cite{friedrich2019entropy} derive $\bm{f}_t$ for the Euler equations,
magnetohydrodynamics equations, and shallow water equations in~\cite{friedrich2019entropy}.

\subsection{Computational domain}

We introduce the space-time domain $\Omega$, comprising of a spatial domain and a temporal interval, $\Omega=\Omega_s\times[0,T]$. The domain is split into $K$ non-overlapping elements $\Omega_m$, $\Omega \simeq \Omega^h \coloneqq \bigcup_{m=1}^K \Omega_m$, introducing the subscript $m$ for information local to an element. For simplicity, we consider that each element is a rectangular Cartesian element with dimensions $\Delta x_m \times \Delta t_m$.
{Furthermore, each element has a boundary $\Gamma_m$, each parallel to either the spatial direction (denoted as $\Gamma_s$) or temporal direction (denoted as $\Gamma_t$).

Each element $\Omega_m$ is related to the reference space-time element $\Omega^r$,
which is defined on the reference coordinates, $\bm{\bar{\xi}} \coloneqq (\xi, \tau): -1 \leq \xi,\tau \leq 1$. Hence, we can represent the space-time reference domain as $\Omega^r = \Omega_s^r \times \Omega_t^r$, with $\Omega_t^r = (-1,1)$; with corresponding boundary $\Gamma^r$ which consist of spatial faces, $\Gamma_s^r \times \Omega_t^r$ and temporal faces $\Omega_s^r \times \Gamma_t^r$ with $\Gamma_t^r = \{-1,1\}$. } The physical coordinates $\bm{\bar{x}} = (x,t)$ of each element are related to the reference coordinates by an invertible mapping function $\bm{\bar{x}}_m (\bm{\bar{\xi}})$.
For the Cartesian $1D+1$ case, the Jacobian determinant of the mapping is
\begin{equation}
    J_m = \frac{1}{4} \Delta x_m \Delta t_m.
\end{equation}
The Jacobian cofactor matrix is
\begin{equation}
    \bm{C}_m = \frac{1}{2} \sqrt{\Delta x_m\Delta t_m} \bm{I}_2,
\end{equation}
where $\bm{I}_2$ is the two-by-two identity matrix. {\color{black}We also define a 1D Jacobian arising in integrations over the face of an element,
\begin{equation}
    J_{1D}= \frac{1}{2} \sqrt{\Delta x_m\Delta t_m}.
\end{equation}
}

The approximate solution is defined as a direct sum of local approximations on non-overlapping elements,
\begin{equation}
    \bm{u}({x},t) \simeq \bm{u}^h({x},t) =\bigoplus_{m=1}^K \bm{u}^h_m({x},t),
\end{equation}
where we consider approximate local solution
\begin{equation}
     \bm{u}^h_m({x},t)  = \sum_{i=1}^{N_{soln}}{{\chi}_{m,i}(\bm{\bar{x}})\hat{{u}}_{m,i}},
     \label{eq:approx_soln}
\end{equation}
where $\hat{u}_{m,i}$ are modal solution coefficients, represented on a set of $N_{soln}$ basis functions ${\chi}_{m,i}$. We omit the superscript $h$ hereafter for the modal solution coefficients in the interest of simple notation.

We permit that the solution and flux may be represented on different bases. One may prefer to use a basis with a higher integration strength, such that $N_{soln}$ and $N_{q}$ are not strictly equal, to represent the flux with reduced aliasing errors. Thus, we introduce the flux basis functions ${\phi}_{m,i}$, such that the $j$-th component of the flux is defined as
\begin{equation}
    {f}_{m,j}({x},t)  = \sum_{i=1}^{N_{q}}{{\phi}_{m,i}(\bm{\bar{x}})\hat{{f}}_{m,j,i}}.
\end{equation}

We solve the PDE~\eqref{eq:general_cons_law} on the reference element. The reference residual on an element is defined as
\begin{equation}
    R_m^{h,r}(\bar{\bm{\xi}}) \coloneqq  \frac{\partial }{\partial \tau}\bm{f}_t\big(\bm{u}(\bar{\bm{x}}(\bar{\bm{\xi}}))) \big)+ {\nabla}^r \cdot {\bm{f}_s}(\bm{u}(\bar{\bm{x}}(\bar{\bm{\xi}})))
    \label{eq:residual}
\end{equation}
For the remainder of this work, we will omit subscript $m$. It is understood that operators and solutions are local.

\subsection{Node and basis function choices}
There is considerable freedom allotted to the selection of basis functions and nodes, which we will describe before proceeding to the space-time discretization.
We make some reasonable assumptions to simplify the scheme. The primary simplification is that we restrict the element types to Cartesian tensor-product elements. 
It is possible to extend the methods herein to curvilinear meshes (cf. \cite{friedrich2019entropy}), but for the test cases in one spatial dimension, we choose to restrict the element types to the simplest Cartesian case.
Tensor product elements enable sum-factorization techniques~\cite{orszag1979spectral}, improving code scaling.
When choosing tensor-product nodes for a space-time discretization, it is possible to select different points in the temporal and spatial directions. In fact, Friedrich \textit{et al.}~\cite{friedrich2019entropy} demonstrate that one can reduce the integration strength of temporal nodes without sacrificing the overall order of accuracy. While their results demonstrate benefits to using unbalanced nodes, we choose the simpler case where elements have the same number and type of nodes in space and time.
The tensor-product basis functions are written as
\begin{equation}
    \bm{\chi}(\bar{\bm{x}}) \coloneqq \bm{\chi}(x) \otimes \bm{\chi}(t), \in \mathbb{R}^{1 \times (N_{soln})},
\end{equation}
where $\bm{\chi}$ are polynomial basis functions in both space and time. The flux basis $\bm{\phi}(\bar{\bm{x}}) \in \mathbb{R}^{1 \times (N_{q})}$ is defined similarly.

We use a reference element with three sets of nodes. In principle, there is considerably more freedom than the restrictions outlined below. For a scheme which we denote $p$-th order, we define:
\begin{itemize}
    \item Solution nodes, $\bar{\bm{\xi}}_{soln}$, which are a set of $N_{soln} = (p+1)^2$ Gauss-Legendre (GL) or Gauss-Lobatto-Legendre (GLL) nodes. 
    \item Flux quadrature nodes, $\bar{\bm{\xi}}_{q}$, which are a set of $N_{q} = (p+1+N_{OI})^2$ GL or GLL nodes, with the possibility of overintegrating the flux by $N_{OI}$.
    \item Face nodes on the $i$-th face, $\bar{\bm{\xi}}_{f,i}$. In a $1D+1$ space-time element, we choose to use the same node type on each face, and further restrict the choice to be the 1D equivalent to the flux nodes such that we have $N_{f} = p+1+N_{OI}$ points on each face. It is natural to use the same quadrature rule as the flux nodes because we use the face nodes to evaluate numerical fluxes.
\end{itemize}
Each quadrature rule admits a set of integration weights, which we store in diagonal matrices $\mathcal{W}_{soln}, \mathcal{W}_{q}$ and $\mathcal{W}_{f,i}$. 
Furthermore, we choose to use basis functions $\bm{\chi}$ and $\bm{\phi}$ which are Lagrange polynomials defined at the solution and flux nodes, respectively. Choosing collocated basis and volume nodes is not strictly required, but results in some operators being diagonal.

Figure \ref{fig:reference_elem} depicts the reference element with faces labelled. In a $1D+1$ element, faces $1$ and $2$ will always be perpendicular to the spatial dimension, such that the solution along faces $3$ and $4$ will be the $1D$ solution at the beginning and end of a time step, respectively.
We consider two categorizations for the face nodes. The first separates nodes on each face, i.e., $\{\bar{\bm{\xi}}_{f,1}, \bar{\bm{\xi}}_{f,2},\bar{\bm{\xi}}_{f,3},\bar{\bm{\xi}}_{f,4} \}$ for each of the four faces, and the second groups all face nodes together, i.e., $\bar{\bm{\xi}}_{f} = [\bar{\bm{\xi}}_{f,1}\ \  \bar{\bm{\xi}}_{f,2} \ \ \bar{\bm{\xi}}_{f,3} \ \ \bar{\bm{\xi}}_{f,4}]$.

\begin{figure}[h!]
    \centering
    \includegraphics[height=2.5in]{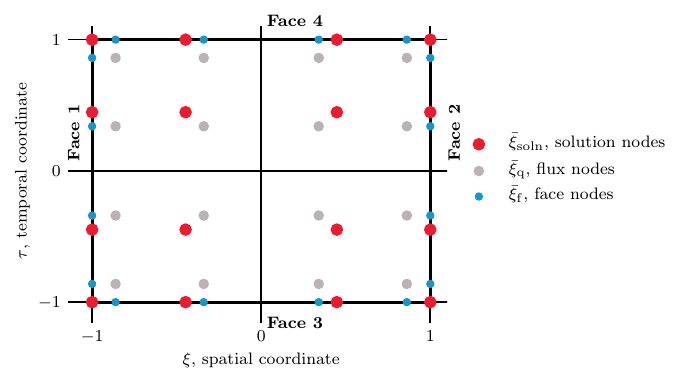}
    \caption{Schematic of the reference space-time element $\Omega_r$. For this $p=3$ example, GLL nodes are chosen for the solution nodes and GL nodes without overintegration are used for the flux and face nodes.}
    \label{fig:reference_elem}
\end{figure}

\subsection{Element Numbering}
\label{sec:elem-numbering}

The element numbering within the domain $\Omega$ is depicted in Fig.~\ref{fig:domain_numbering}. We will refer to elements having the same temporal index as a ``timeslab", e.g., the elements grouped by the dashed line in Fig.~\ref{fig:domain_numbering} make up the first timeslab.

\begin{figure}[h!]
    \centering
    {\setlength{\fboxsep}{10pt}\fbox{
    \begin{tikzpicture}[scale=1.2]

    \tikzstyle{box} = [rectangle, draw, minimum size=0.85cm, thick]
    \tikzstyle{labelstyle} = [font=\small]

    
    \node[box] (b11) at (0,0) {};
    \node[box] (b12) at (1,0) {};
    \node at (2,0) {$\dots$};
    \node[box] (b1k) at (3,0) {};

    \node[box] (b21) at (0,1) {};
    \node[box] (b22) at (1,1) {};
    \node at (2,1) {$\dots$};
    \node[box] (b2k) at (3,1) {};

    \node at (0,1.8) {$\vdots$};
    \node at (1,1.8) {$\vdots$};
    \node at (3,1.8) {$\vdots$};

    \node[box] (bk1) at (0,2.5) {};
    \node[box] (bk2) at (1,2.5) {};
    \node at (2,2.5) {$\dots$};
    \node[box] (bkk) at (3,2.5) {};


    \node[labelstyle] at (0,-0.8) {1};
    \node[labelstyle] at (1,-0.8) {2};
    \node[labelstyle] at (2,-0.8) {$\dots$};
    \node[labelstyle] at (3,-0.8) {$K_s$};
    
    \draw[->, thick] (-0.4,-1.1) -- (1,-1.1) node[midway, below] {$x$};

    \node[labelstyle] at (-0.8, 0) {1};
    \node[labelstyle] at (-0.8, 1) {2};
    \node[labelstyle] at (-0.8, 1.8) {$\vdots$};
    \node[labelstyle] at (-0.8, 2.5) {$K_t$};

    \draw[->, thick] (-1.3, -0.4) -- (-1.3, 1) node[midway, left, rotate=90, anchor=south] {$t$};

    \draw[dashed, thick] ($(b11.south west) + (-0.15,-0.15)$) rectangle ($(b1k.north east) + (0.15,0.15)$);

\end{tikzpicture}}}
    \caption{Schematic of element numbering. Elements are numbered in the $x$ direction from $1$ to $K_s$. Timeslabs are numbered from $1$ to $K_t$, with dashes indicating the first timeslab.}
    \label{fig:domain_numbering}
\end{figure}
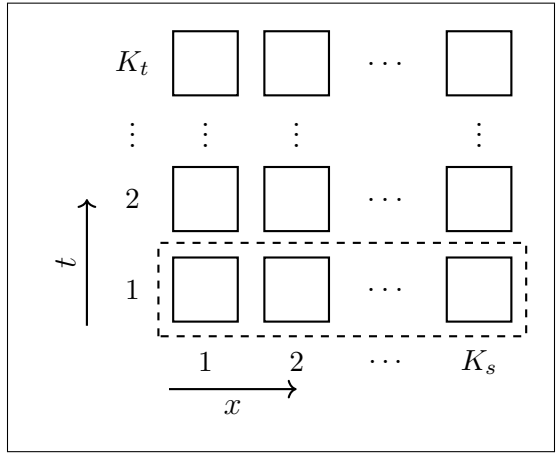

\subsection{Discontinuous Galerkin Discretization}
\label{sec:DG}
The ST-ESFR discretization is formulated with
energy stability in a broken Sobolev norm instead of the standard DG norm.
We begin the DG discretization by multiplying the elementwise residual \eqref{eq:residual} by a set of $N_{soln}$ orthogonal test functions {-- here the flux basis $\phi$ -- }choosing the same functions as the basis functions. The result is integrated over the reference element and integrated by parts twice, arriving at the integral strong form: {
\begin{equation}
\begin{aligned}
    \int_{\Omega^r} \phi_{i}(\bar{\bm{\xi}})
    &\left(
    \frac{\partial \phi_{j}(\bar{\bm{\xi}}) }{\partial \tau}\bm{f}_{t,j}\big(\bm{u}(\bar{\bm{x}}(\bar{\bm{\xi}}))) \big)+ \frac{\partial \phi_{j}(\bar{\bm{\xi}}) }{\partial \xi} {\bm{f}_{s,j}}(\bm{u}(\bar{\bm{x}}(\bar{\bm{\xi}})))
    \right)
    \ d\Omega^r \\
    &+ \int_{\Gamma^r} \phi_i(\bar{\bm{\xi}}) 
         \left[ 
            \hat{\bm{{n}}}_t  \cdot \left(\bar{\bm{f}}^{*,r}_{t,j} - \phi_j (\bar{\bm{\xi}}) \hat{{\bm{f}}}^{r,T}_{t,j}\right)  
            +\hat{\bm{{n}}}_s  \cdot \left(\bar{\bm{f}}^{*,r}_{s,j} - \phi_j (\bar{\bm{\xi}}) \hat{{\bm{f}}}^{r,T}_{s,j}\right)
        \right] \ d \Gamma^r
        =0, \\
        &\ \ \ \ \ \ \ \ \ \ \ \ \ \ \ 
        \forall 
        \begin{cases}
            i= 1, ..., N_{soln} \\
            j = 1, ..., N_{soln}
        \end{cases}
\end{aligned}
\end{equation}
where $\hat{\bm{{n}}}_t$ is the $t$ component of a unit normal vector along the surface of the element and likewise for $\hat{\bm{{n}}}_s$.}
We evaluate the integrals using a set of appropriate quadrature points $\bar{\bm{\xi}}$ and quadrature weights stored as a diagonal matrix $\mathcal{W}$. Each integral is evaluated using its own integration rule, resulting in the distinct subscripts $f$ and $q$.
In preceding work~\cite{allaneau2011connections,zwanenburg2016equivalence,cicchino2021new,cicchino2022nonlinearly,cicchino2022provably,cicchino2025discretely}, FR is formulated 
through a modification of the DG mass matrix. 
As we wish to use a similar mechanism, we premultiply the discretization by an inverse mass matrix to lead to an operator form of the local discretization,
\begin{equation}
     \mathcal{D}_t \hat{\bm{f}}^{T}_{t}
     - L_t \left( \bm{\phi}(\bar{\bm{\xi}}_{f}) \hat{\bm{f}}_{t}^{r,T} - \bm{f}^{*,r}_{t,f}
    \right) +
     \mathcal{D}_s \hat{\bm{f}}^{T}_{s} 
        - L_s \left( \bm{\phi}(\bar{\bm{\xi}}_{f}) \hat{\bm{f}}_{s}^{r,T} - \bm{f}^{*,r}_{s,f} \right)
    = \bm{0}^T,
    \label{eq:DG-local}
\end{equation}
where the local differential and lifting operators in the temporal dimension are defined 
\begin{equation}
    \mathcal{D}_t = \mathcal{M}^{-1}\bm{\chi}(\bar{\bm{\xi}}_{q})^{T} \mathcal{W}_{q} \frac{\partial \bm{\phi} (\bar{\bm{\xi}}_{q})}{\partial\tau}, \ \ \ \ \ \ \ 
    {L_t} = \mathcal{M}^{-1} \bm{\chi}(\bar{\bm{{\xi}}}_{f})^T \mathcal{W}_{f} \text{diag}(\hat{\bm{n}}_t)
\end{equation}
using a mass matrix evaluated using the solution basis and flux quadrature points,
\begin{equation}
    \mathcal{M} = \bm{\chi}(\bar{\bm{\xi}}_{q}) ^T \mathcal{W}_{q} \bm{J} \bm{\chi} (\bar{\bm{\xi}}_{q}),
    \label{eq:mass matrix}
\end{equation}
with $\bm{J}$ being a diagonal matrix having entries of the constant Jacobian $J$.
The spatial local differential operator, $\mathcal{D}_s$, and lifting operator, $L_s$, are defined alike the temporal ones such that, in the present DG case, the premultiplication by the inverse mass matrix does not impact the solution.

As we solve on the reference domain, the physical flux is evaluated on the flux nodes and transformed to the reference element,
\begin{equation}
    \hat{\bm{f}_t}^r(\bm{\hat{u}}) = J_{1D}\Pi_{q} \bm{f}_t(\bm{\chi}(\bm{\bar{\xi}}_{q}) \hat{\bm{u}}) 
\end{equation}
where $\bm{f}_t()$ is the physical flux function and $J_{1D}$ arises due to the cofactor matrix. The same treatment is applied to temporal and spatial fluxes. The projection operator is defined
\begin{equation}
    {\Pi}_{q} = (\mathcal{M})^{-1} \bm{\chi}(\bm{\xi}_{q})^{T} \mathcal{W}_{q}.
    \label{eq:proj_flux}
\end{equation}

\section{Space-Time Energy Stable Flux Reconstruction for Linear Problems}
\label{sec:ST-ESFR}

The FR discretization redefines the spatial operators to arrive at a form similar to the DG discretization of Eq. (\ref{eq:DG-local}), 
\begin{equation}
       \mathcal{D}_t \hat{\bm{f}}^{r,T}_{t}
     - L_t \left( \bm{\phi}(\bar{\bm{\xi}}_{f}) \hat{\bm{f}}_{t}^{r,T} - \bm{f}^{*,r}_{t,f}
    \right)  + \mathcal{D}_{FR,s} \hat{\bm{f}}^{r,T}_{s} 
        - L_{FR,s} \left( \bm{\phi}(\bar{\bm{\xi}}_{f}) \hat{\bm{f}}_{s}^{r,T} - \bm{f}^{*,r}_{s,f} \right)
    = \bm{0}^T,
    \label{eq:ST-ESFR}
\end{equation}
while the temporal operators remain identical to the definitions given in Section \ref{sec:DG}. The FR operators, $\mathcal{D}_{FR}$ and ${L}_{FR}$, will be defined later in this section.

The FR modification begins with the introduction of $p+1$ one-dimensional correction functions $\bm{g}$ in the spatial coordinate following the foundations of Huynh~\cite{huynh_flux_2007}. The correction functions fulfill a symmetry condition $g^L(\xi) = -g^R(-\xi)$. We further require that each correction function satisfies
\begin{equation}
    g^{f} (\xi_{f_i}) \hat{n}_{s,f_i} = 
    \begin{cases}
        1 \text{ if } f=f_i \\
        0 \text{ otherwise},
    \end{cases}
\end{equation}
on the $f$-th face. That is, the magnitude of the correction functions go to unity at one face, and zero at the opposite face. Recalling that $\hat{n}_s=0$ on faces $3$ and $4$, this definition only applies to faces $1$ and $2$, which are perpendicular to the spatial direction.
According to the fundamental assumption of ESFR~\cite{vincent_new_2011}, a FR scheme is ESFR iff the one-dimensional FR correction functions for an energy-stable scheme are chosen according to
\begin{equation}
    \int_{\Omega_r}\frac{\partial \chi_i(\xi)}{\partial \xi} g^{f} (\xi) \ d \Omega_r 
    - c \frac{\partial^p \chi_i(\xi)^T}{\partial \xi^p}\frac{\partial^{p+1} g^{f} (\xi)}{\partial \xi^{p+1}}
    =0, \forall i = 1, ..., N_p. 
\end{equation}
Under such an assumption, ESFR methods are unified by the correction parameter $c$. The space-time correction functions are uniform in time, such that the temporal direction is not corrected.
{\color{black}As demonstrated by~\cite{allaneau2011connections,zwanenburg2016equivalence,cicchino2021new}, the influence of flux reconstruction can be expressed as a filter on a DG scheme. The ESFR filter operator in a single dimension, $\mathcal{K}_{1D}$, is defined alike previous literature,}
\begin{equation}
    (\mathcal{K}_{1D})_{i,j} \approx c \int_{\Omega_r} J^\Omega \frac{\partial^p \chi_i(\xi)}{\partial \xi^p}\frac{\partial^p \chi_j(\xi)}{\partial \xi^p} \ d\Omega_r  \ \ \ \ \ \ \longrightarrow \ \mathcal{K}_{1D} = c(\mathcal{D}_{1D}^p)^T \mathcal{M}_{1D} (\mathcal{D}_{1D}^p).\label{eq:K_1D}
\end{equation}
It modifies the strong DG mass matrix according to a correction parameter $c$, with $\mathcal{D}^p$ denoting a differential operator constructed on a 1D face raised to the $p$-th power,
\begin{equation}
    \mathcal{D}_{1D}^p = \left(\mathcal{M}_{1D}^{-1}\bm{\chi}({\bm{\xi}}_{f})^{T} \mathcal{W}_{f} \frac{\partial \bm{\chi} ({\bm{\xi}}_{f})}{\partial\xi}\right)^p,\label{eq:D_1D}
\end{equation}
and $\mathcal{M}_{1D}$ being a one-dimensional equivalent to \eqref{eq:mass matrix} without Jacobian influence.
{ We apply the FR correction functions only in the spatial dimension, so the 1D+1 FR operator $\mathcal{K}$ is extended from Eq.~\eqref{eq:K_1D} as a tensor product,}
\begin{equation}
    \mathcal{K} = \frac{1}{J_{1D}}c(\mathcal{D}_\xi^p)^T \mathcal{M} (\mathcal{D}_\xi^p) = \mathcal{K}_{1D} \otimes \mathcal{M}_{1D},
    \label{eq:K_matrix}
\end{equation}
{with $\mathcal{D}_\xi^p$ is the $1D+1$ equivalent to Eq.~\eqref{eq:D_1D}.}
\begin{remark}
    The definition of the FR matrix as stated in Eq.~\eqref{eq:K_matrix} differs from previous multi-dimensional works {\color{black} where the FR modification impacts all coordinate directions,} cf. ~\cite[Eq. (39)]{cicchino2022provably}. The present scheme can be viewed as ESFR in space and DG in time.
\end{remark}

{We modify the mass matrix alike~\cite{cicchino2022provably,cicchino2022nonlinearly,cicchino2025discretely} to apply FR in the spatial dimension to form the FR differential and lifting operators,}
\begin{equation}
    \mathcal{D}_{FR,\xi} = ({\mathcal{M} + \mathcal{K}})^{-1}\bm{\chi}(\bar{\bm{\xi}}_{q})^{T} \mathcal{W}_{q} \frac{\partial \bm{\phi} (\bar{\bm{\xi}}_{q})}{\partial\xi}
    \end{equation}
    and 
    \begin{equation}
    {L_{FR,\xi}} = ({\mathcal{M} + \mathcal{K}})^{-1} \bm{\chi}(\bar{\bm{{\xi}}}_{f})^T \mathcal{W}_{f} \text{diag}(\hat{\bm{n}}_s),
\end{equation}
where $(\mathcal{M+K})$ is referred to as the FR mass matrix. 
{The temporal operators remain unchanged from the DG scheme.}

\subsection{Discrete Conservation for Linear Advection}

We consider a periodic $1D+1$ linear advection problem to investigate the properties of the FR discretization,
\begin{equation}
    \begin{cases}
        \frac{\partial u}{\partial t} + a \frac{\partial u}{\partial x} = 0 \ \ \ \forall (x, t) \in ([x_L, x_R]\times[0, T])\\
        u(x_L,t) = u(x_R,t)\\
        u(x, 0) = u_0(x)
    \end{cases}
\end{equation}
where $a$ is a constant advection speed, and proceed to a discrete conservation proof.
We first introduce a property of the {FR mass matrix} that will aid the proof.

\begin{lemma} \label{lem:filter-of-1}
    The statement $ \mathcal{K} \hat{\bm{1}}^T = \hat{\bm{0}}^T$ holds for the FR portion of the mass matrix $\mathcal{K}$ as introduced in Section~\ref{sec:ST-ESFR}.
\end{lemma}
\begin{proof}
    {Consider that we define $\bm{\chi}(\bm{\bar{\xi}}_{q})\hat{\bm{1}}^T = \bm{1}^T$.} According to its definition, $\mathcal{K}$ is an operator taking the $p$-th derivative of a vector of solution coefficients in the spatial direction. $\bm{\hat{1}}^T$ holds the coefficients of a $0$-th order polynomial, so the $p$-th derivative will vanish. Therefore, the equality 
    $ \mathcal{K} \bm{\hat{1}}^T = \hat{\bm{0}}^T$ holds.
\end{proof}

\begin{definition} \label{def:IBP}
    In the notation introduced in Sections~\ref{sec:prelim} and \ref{sec:ST-ESFR}, the summation-by-parts (SBP) property holds as follows in the spatial direction,
    \begin{equation}
     \mathcal{D}_{FR,s} + \mathcal{D}_{FR,s}^T = L_{FR,s,1}\bm{\phi}(\bar{\bm{\xi}}_{f,1}) + L_{FR,s,2} \bm{\phi}(\bar{\bm{\xi}}_{f,2}),
\end{equation}
while for the temporal direction,
\begin{equation}
     \mathcal{D}_t + \mathcal{D}_t^T = L_{t,3}\bm{\phi}(\bar{\bm{\xi}}_{f,3}) + L_{t,4} \bm{\phi}(\bar{\bm{\xi}}_{f,4})
     .
\end{equation}
\end{definition}

\begin{theorem}
\label{thm:ESFR-conservation}
The ST-ESFR scheme is discretely globally conservative for general $c$.
\end{theorem}

\begin{proof}
We begin the proof by rearranging Eq.~\eqref{eq:ST-ESFR} on a single element for the linear advection case where the physical flux is $f_s=au$,
\begin{equation} 
\begin{aligned}  
 - \mathcal{D}_t J_{1D} \hat{\bm{u}}^T
    &+ \bm{L}_{t}\left( \bm{\phi}(\bar{\bm{\xi}}_{f}) J_{1D}\hat{\bm{u}}^T - \bm{f}^{*,r}_{t}) \right) \\
    &- a \mathcal{D}_{FR,s} J_{1D} \hat{\bm{u}}^T 
     +
    \bm{L}_{FR,s} \left( a \bm{\phi}(\bar{\bm{\xi}}_{f}) J_{1D} \hat{\bm{u}}^T -  \bm{f}_{s}^{*,r} \right) 
    = \bm{0} 
\end{aligned}
\label{eq:general_one_element}
\end{equation}

We then left-multiply by $\bm{\hat{1}} (\mathcal{M}+\mathcal{K})$, which signifies integration in the broken Sobolev norm, and use the SBP property per definition~\ref{def:IBP} in each direction to remove terms, resulting in
\begin{equation} 
\begin{aligned}
    -  \hat{\bm{1}}  (\mathcal{M} + \mathcal{K}) \mathcal{D}_t^T J_{1D}\hat{\bm{u}}^T
    & +  \hat{\bm{1}}  (\mathcal{M}+\mathcal{K}) L_{t} \left( - \bm{f}^{*,r}_{t} \right) \\
    & - a \hat{\bm{1}} (\mathcal{M}+\mathcal{K}) \mathcal{D}_{FR,s}^T J_{1D} \hat{\bm{u}}^T + 
    \hat{\bm{1}} (\mathcal{M}+\mathcal{K}) L_{FR,s} \left(  -  \bm{f}_{s}^{*,r} \right) = \bm{0}.
    \end{aligned}
\end{equation}
Being scalars, all terms can be transposed. We use Lemma~\ref{lem:filter-of-1} to show that terms having $\mathcal{K}\hat{\bm{1}}$  vanish after transposing. The remaining parts of the volume terms also vanish as they take the derivative of a constant. Therefore, only surface terms remain. Writing terms only on the faces where they are nonzero, we have
\begin{equation} 
\begin{aligned}
 +  \hat{\bm{1}}  (\mathcal{M}) L_{t,3} \left( - \bm{f}^{*,r}_{t,3}) \right) 
    & + \hat{\bm{1}}  (\mathcal{M}) L_{t,4} \left( - \bm{f}^{*,r}_{t,4}\right)\\
    + \hat{\bm{1}} (\mathcal{M}) L_{s,1} \left(  -  \bm{f}_{s,1}^{*,r} \right)
    &+ 
    \hat{\bm{1}} (\mathcal{M}) L_{s,2}\left(- \bm{f}_{s,2}^{*,r} \right)
     = \bm{0}.
    \end{aligned}
    \label{eq:thm1-terms-on-face}
\end{equation}

We now consider a single timeslab as presented in Section~\ref{sec:elem-numbering}. 
An appropriate choice of numerical fluxes on a periodic domain causes the first two terms of Eq.~\eqref{eq:thm1-terms-on-face} to cancel, leaving only the surface contributions on the $3$ and $4$ faces. Expanding the lifting operator, the energy balance on the timeslab is
\begin{equation} 
    \sum_{k=1}^{K_s}\hat{\bm{1}} \bm{\chi}(\bar{\bm{\xi}}_{f,3})^T \mathcal{W}_{f,3}  \bm{f}^{*,r}_{t,3}
     = \sum_{k=1}^{K_s} \hat{\bm{1}}
    \bm{\chi}(\bar{\bm{\xi}}_{f,4})^T \mathcal{W}_{f,4}\bm{f}^{*,r}_{t,4}.
\label{eq:conservation_time_linadv}
\end{equation}
We omit the summation index for notational brevity.
Now considering the entire computational domain, numerical fluxes will cancel on interior faces between timeslabs.
For the global case, we point out that the only well-posed choice for temporal numerical flux at the $t=0$ face is upwind, corresponding to a Dirichlet initial condition, and upwind at the $t=T$ face, corresponding to an outflow boundary. Therefore, \eqref{eq:conservation_time_linadv} extends to global conservation in the sense that
\begin{equation}
\begin{aligned}
         \sum_{k=1}^{K_s} \hat{\bm{1}} \bm{\chi}(\bar{\bm{\xi}}_{f,3})^T  &\mathcal{W}_{f,3}  J_{1D} \bm{u}_0(\bm{x}(\bar{\bm{\xi}}_{f,3}))  \Bigg|_{t=0}\\
     &= \sum_{k=1}^{K_s} \hat{\bm{1}}
    \bm{\chi}(\bar{\bm{\xi}}_{f,4})^T \mathcal{W}_{f,4}  J_{1D} \bm{\phi} (\bar{\bm{\xi}}_{f,4})\hat{\bm{u}} \Bigg|_{t=T}.
\end{aligned}
    \label{eq:lin-adv-global-cons}
\end{equation}

\end{proof}

\begin{remark}
    While the $t=0$ surface requires a Dirichlet boundary and the $t=T$ must use pure outflow, the interior surfaces between timeslabs may use other numerical fluxes while maintaining conservation.
\end{remark}

\subsection{Numerical Experiments Using the Space-time ESFR Scheme}
\label{sec:lin-adv-results}

We use linear advection with an advection speed of $a=0.6$ to verify the properties of the ST-ESFR, with initial condition
\begin{equation}
    u_0(x) = 2 \sin(\pi x) + 1.01 
\end{equation}
on the domain $[0,2] \times [0,2]$. We use an upwind numerical flux in space. 
The implicit system is solved using a Jacobian-free Newton-Krylov solver~\cite{knoll2004jacobian} written using the GMRES solver of \texttt{IterativeSolvers.jl}~\cite{Chen2026JuliaLinearAlgebra}. We use a nonlinear tolerance of $1E-10$, which we find to be small enough to be conservative to machine precision.
The $L_2$ error is calculated by comparing to the exact solution $u_{\text{exact}}(x,t)=2 \sin(\pi (x-0.6t)) + 1.01$ across the entire space-time domain and overintegrating by 10 quadrature points per dimension. 

We present the $L_2$ error magnitude and convergence in Table \ref{tab:conv-lin-adv-GLLGL} and Table~\ref{tab:conv-lin-adv-GLGL} using uncollocated and collocated nodes, respectively. We have verified that the $L_\infty$ error follows similar trends. Due to the linearity of the fluxes, both of the node choices lead to the same errors at the level of the linear solver precision. For both $c_{DG}$ and $c_{HU}$, the space-time discretization converges at $p+1$ for even and odd $p$.

\begin{table}[h!]
    \centering
        \caption{Convergence of the ST-ESFR method using GLL solution nodes and GL flux nodes on the linear advection test case.}
        \label{tab:conv-lin-adv-GLLGL}
\begin{tabular}{|c|c|ll|ll|}
\hline
\multirow{2}{*}{Convergence test} & \multirow{2}{*}{$N$} & \multicolumn{2}{c|}{$c_{DG}$} & \multicolumn{2}{c|}{$c_{Hu}$} \\
                                  &                      & $L_2$ error       & rate      & $L_2$ error       & rate      \\ \hline
                                  & 2                    & 8.10E-02          & --        & 1.46E-01          & --        \\
                                  & 4                    & 5.15E-03          & 3.97      & 1.10E-02          & 3.73      \\
                                  & 8                    & 3.27E-04          & 3.98      & 7.09E-04          & 3.96      \\
$p=3$                             & 16                   & 2.04E-05          & 4.00      & 4.46E-05          & 3.99      \\
                                  & 32                   & 1.27E-06          & 4.00      & 2.79E-06          & 4.00      \\
                                  & 64                   & 7.96E-08          & 4.00      & 1.75E-07          & 4.00      \\
                                  & 128                  & 4.97E-09          & 4.00      & 1.09E-08          & 4.00      \\ \hline
                                  & 2                    & 1.12E-02          & --        & 2.28E-02          & --        \\
                                  & 4                    & 4.01E-04          & 4.81      & 8.15E-04          & 4.80      \\
                                  & 8                    & 1.27E-05          & 4.98      & 2.62E-05          & 4.96      \\
$p=4$                             & 16                   & 3.96E-07          & 5.00      & 8.21E-07          & 5.00      \\
                                  & 32                   & 1.24E-08          & 5.00      & 2.56E-08          & 5.00      \\
                                  & 64                   & 3.86E-10          & 5.00      & 8.00E-10          & 5.00      \\
                                  & 128                  & 1.22E-11          & 4.99      & 2.50E-11          & 5.00      \\ \hline
\end{tabular}
    \label{tab:placeholder}
\end{table}

\begin{table}[h!]
\centering
\caption{Convergence of the ST-ESFR method using collocated GL solution and flux nodes on the linear advection test case.}
\label{tab:conv-lin-adv-GLGL}
\begin{tabular}{|c|c|ll|ll|}
\hline
\multirow{2}{*}{Convergence test} & \multirow{2}{*}{$N$} & \multicolumn{2}{c|}{$c_{DG}$} & \multicolumn{2}{c|}{$c_{Hu}$} \\
                                  &                      & $L_2$ error       & rate      & $L_2$ error       & rate      \\ \hline
                                  & 2                    & 8.10E-02          & --        & 1.46E-01          & --        \\
                                  & 4                    & 5.15E-03          & 3.97      & 1.10E-02          & 3.73      \\
                                  & 8                    & 3.27E-04          & 3.98      & 7.09E-04          & 3.96      \\
$p=3$                             & 16                   & 2.04E-05          & 4.00      & 4.46E-05          & 3.99      \\
                                  & 32                   & 1.27E-06          & 4.00      & 2.79E-06          & 4.00      \\
                                  & 64                   & 7.96E-08          & 4.00      & 1.75E-07          & 4.00      \\
                                  & 128                  & 4.97E-09          & 4.00      & 1.09E-08          & 4.00      \\ \hline
                                  & 2                    & 1.12E-02          & --        & 2.28E-02          & --        \\
                                  & 4                    & 4.01E-04          & 4.81      & 8.15E-04          & 4.80      \\
                                  & 8                    & 1.27E-05          & 4.98      & 2.62E-05          & 4.96      \\
$p=4$                             & 16                   & 3.96E-07          & 5.00      & 8.21E-07          & 5.00      \\
                                  & 32                   & 1.24E-08          & 5.00      & 2.56E-08          & 5.00      \\
                                  & 64                   & 3.86E-10          & 5.00      & 8.00E-10          & 5.00      \\
                                  & 128                  & 1.20E-11          & 5.00      & 2.50E-11          & 5.00      \\ \hline
\end{tabular}
\end{table}

A key property of ESFR schemes is that the optimal $p+1$ order of convergence is maintained as the correction parameter $c$ is increased, until the order of accuracy drops to $p$~\cite{castonguay2012high}. 
We use a numerical convergence order obtained by performing a single refinement from $16\times16$ elements to $32\times32$ elements using a range of $c$ values from $10^{-7}$ to $10^{4}$. 
We compare to a MOL result from the same code, wherein the FR spatial part of the scheme~\eqref{eq:ST-ESFR} is paired with the RK54 method. 
A small time step size is chosen such that the spatial error dominates. 
Figure \ref{fig:OOA-vs-c} demonstrates that both method-of-lines and space-time FR schemes lose an order at a similar location. 
{\color{black}All four combinations of GL and GLL nodes were tested, and none yielded visual differences; we therefore plot only the result using the uncollocated GLL solution and GL quadrature nodes.}
Furthermore, they lose an order at approximately the same $c$ value as that of Castonguay~\cite{castonguay2012high}; the slight deviation can be attributed to the difference in the temporal integration method.

\begin{figure}[h!]
    \centering
    \includegraphics[height = 4 in]{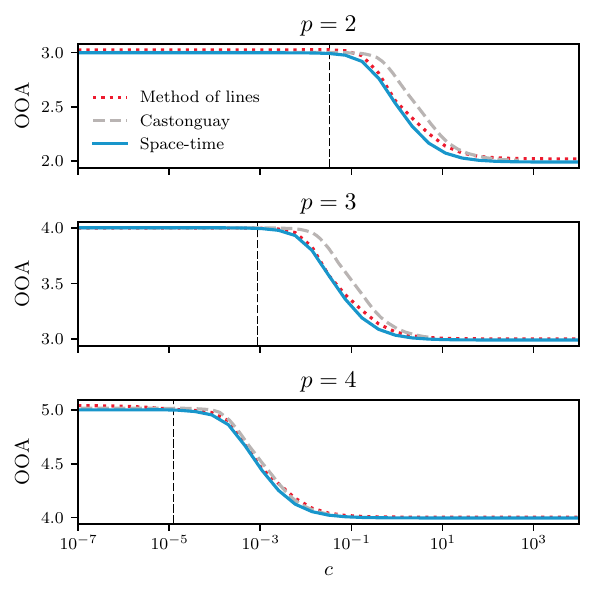}
    \caption{Experimental convergence order for the method of lines and space-time ESFR approaches for linear advection using the same code. The result is compared to reference data from Castonguay~\cite[Figure 3.6]{castonguay2012high}, noting that $c$ values have been adjusted to a Legendre basis per \cite[Section 3]{cicchino2021new}. The vertical dashed line indicates $c_{Hu}$.}
    \label{fig:OOA-vs-c}
\end{figure}

\section{Space-Time Entropy Stable Schemes}
\label{sec:ST-NSFR}

The ST-NSFR version of the discretization follows directly from the preliminaries in Section~\ref{sec:prelim} and is similar in philosophy to the ST-ESFR discretization presented in Section~\ref{sec:ST-ESFR}. We begin by defining the spaces arising in the analysis of the ST-NSFR scheme.

\begin{definition}
The classical Sobolev space is defined as 
$$W^{k,p}(\Omega_s) := \{ \bm{u} \in [L^2(\Omega_s)]^{d\times N_s} : \partial^{\alpha} \bm{u} \in [L^2(\Omega_s)]^{d\times N_s}, |\alpha| \leq k \}$$
where $d$ is the spatial domain $\Omega_s \subset \mathbb{R}^{d}$ and $N_s$ denotes the number of components of the vector-valued function $\bm{u}$, with the norm 
    \begin{displaymath}
        \vert\vert \bm{u} \vert\vert_{W^{k,p}(\Omega_s)}^p = \sum_{|\alpha| \leq k}\int_{\Omega_s} \left(\partial^{\alpha} \bm{u} \right)^p\; d\xi,
    \end{displaymath}
and the corresponding seminorm of order $k$
    \begin{equation}
        \vert \bm{u} \vert_{W^{k,p}(\Omega_s)}^p = \sum_{|\alpha| = k}\int_{\Omega_s} \left(\partial^{\alpha} \bm{u} \right)^p\; d\xi.
        \label{eq:k-seminorm}
    \end{equation}

\end{definition}

\begin{definition}
In the case $p=2$, the Sobolev space forms a Hilbert space, with $W^{k,2} = H^k$, where   the Hilbert space $H^k$ admits an inner product and is defined in terms of the $L^2$ inner product
\begin{displaymath}
    \langle \bm{u}, \bm{v} \rangle_{H^{k}(\Omega_s)} = \sum_{|\alpha| \leq k} \langle \partial^{\alpha} \bm{u}, \partial^{\alpha} \bm{v} \rangle_{L^2(\Omega_s)}.
\end{displaymath}
\end{definition}

\begin{definition}
    We define a broken Sobolev space as one that contains only the weighted sum of the first and last in the sequence
        \begin{displaymath}
        \vert\vert \bm{u} \vert\vert_{W^{k,p}_c(\Omega_s)}^p = \int_{\Omega_s} \bm{u}^p + c\left(\partial^{\alpha} \bm{u} \right)^p\; d\xi,
    \end{displaymath}
    with the parameter $c$ as the weight.
    \label{defineBrokenSobolevNorm}
\end{definition}

\begin{definition}
    For both linear and nonlinear conservation laws, the semi-discrete flux reconstruction approach as presented in the previous section establishes energy or nonlinear/entropy stability in the stated broken Sobolev space $H^{k}_c(\Omega)$ with inner product $\langle \cdot, \cdot \rangle$, inducing the norm, 
    \begin{displaymath}
        \frac{\text{d}}{\text{d}t} \vert\vert \bm{u} \vert\vert_{H^{k}_c(\Omega_s)}^2 = \frac{\text{d}}{\text{d}t} \int_{\Omega} \bm{u}^2\ +c\left(\frac{\partial^k \bm{u}}{\partial \xi^k} \right)^2\; d\xi \leq 0,
    \end{displaymath}
    where the modified norm only contains the zeroth and $p$-th derivative terms. The subscript $c$ in $H^{k}_c(\Omega_s)$ signifies the use of the broken Sobolev space. 
\end{definition}

\subsection{Space-time Nonlinearly-Stable Flux Reconstruction Discretization}

As in Section~\ref{sec:ST-ESFR}, the ST-NSFR scheme is formulated by discretizing space and time alike the semidiscretization of~\cite{cicchino2025discretely}. The ESFR mass matrix of Eq.~\eqref{eq:K_matrix} is only applied to the spatial dimension. We proceed to the presentation of the fully-discrete form of the ST-NSFR scheme for a single state on a single element,
{\color{black}\begin{equation}
\begin{aligned}
(\mathcal{M})^{-1}  &
     [\bm{\chi}(\bar{\bm{\xi}}_{q})^T \ \bm{\chi}(\bar{\bm{\xi}}_{f})^T]
    \left[
    (\tilde{\mathcal{Q}} - \tilde{\mathcal{Q}}^T)_{t}
    \odot 
    \mathcal{F}^r_{t} \right]
    \bm{1}^T + {L}_{t} \left( \bm{f}^{*,r}_t
    \right) \\ &+ 
    (\mathcal{M}+\mathcal{K})^{-1}  
    [\bm{\chi}(\bar{\bm{\xi}}_{q})^T \ \bm{\chi}(\bar{\bm{\xi}}_{f})^T ]
    \left[
    (\tilde{\mathcal{Q}} - \tilde{\mathcal{Q}}^T)_{s}
    \odot 
    \mathcal{F}^r_{s} \right] 
    \bm{1}^T \\
     &+ {L}_{FR,s}\left( \bm{f}^{*,r}_s \right)
    = \bm{0}^T.
\end{aligned} \label{eq:STNSFR}
\end{equation} }
The discretization uses the skew-symmetric stiffness operator of Chan~\cite{chan2018discretely}, which is written in the spatial direction as
\begin{equation}
    (\tilde{\mathcal{Q}} - \tilde{\mathcal{Q}}^T)_{s}
    =
    \begin{bmatrix}
        \mathcal{W}_{q} \frac{\partial}{\partial \xi} \bm{\phi}(\bar{\bm{\xi}}_{q}) - \frac{\partial}{\partial \xi} \bm{\phi}(\bar{\bm{\xi}}_{q})^T \mathcal{W}_{q}  
        &  \bm{\phi}(\bar{\bm{\xi}}_{f})^T \mathcal{W}_{f} \text{diag} (\hat{\bm{n}}_s)
         \\
        -\mathcal{W}_{f} \text{diag} (\hat{\bm{n}}_s) \bm{\phi}(\bar{\bm{\xi}}_{f})
         & \bm{0}\\
    \end{bmatrix}
\end{equation}
with $(\tilde{\mathcal{Q}} - \tilde{\mathcal{Q}}^T)_{s} \in \mathbb{R}^{(N_{q} + 4 N_{f})\times(N_{q} + 4 N_{f})} $. The temporal equivalent $(\tilde{\mathcal{Q}} - \tilde{\mathcal{Q}}^T)_{t}$ is defined similarly.
For a constant-metric element, the two-point spatial reference flux is
\begin{equation}
    (\mathcal{F}_s^r)_{ij} = \bm{f}^r_s (\tilde{\bm{u}}(\bar{\bm{\xi}}_i), \tilde{\bm{u}}(\bar{\bm{\xi}}_j)) \ \ \ 
    \forall \  1 \leq i,j \leq N_{q}+N_f N_{f},
\end{equation}
and likewise for $\mathcal{F}_t^r$,
where projections between solution nodes and flux or face nodes are done according to the entropy variables, such that
\begin{equation}
    \tilde{\bm{u}}(\bar{\bm{\xi}}) = \bm{u} \left( \bm{\chi}(\bar{\bm{\xi}}) \hat{\bm{v}}^T \right),
    \ \ \ 
    \hat{\bm{v}}^T = \Pi_{q} \bm{v} \left(\bm{\chi}( \bar{\bm{\xi}}_{q}) \hat{\bm{u}}^T \right),
    \label{eq:entropy-projection}
\end{equation}
using a projection operator defined alike Eq.~\eqref{eq:proj_flux} on the solution nodes.
We define skew-symmetric operator $\tilde{\mathcal{Q}}$ and boundary operator $\mathcal{B}$, which arise in subsequent analyses, as
\begin{equation}
    \tilde{\mathcal{Q}}_s = \frac{1}{2}\begin{bmatrix}
        \mathcal{W}_{q} \frac{\partial}{\partial \xi} \bm{\phi}(\bar{\bm{\xi}}_{q}) - \frac{\partial}{\partial \xi} \bm{\phi}(\bar{\bm{\xi}}_{q})^T \mathcal{W}_{q} 
        & \bm{\phi}(\bar{\bm{\xi}}_{f})^T \mathcal{W}_{f} \text{diag} (\hat{\bm{n}}_s) \\
        - \mathcal{W}_{f} \text{diag} (\hat{\bm{n}_s}) \bm{\phi}(\bar{\bm{\xi}}_{f}) 
        & \mathcal{W}_{f}\text{diag} (\hat{\bm{n}}_s)
    \end{bmatrix}
\end{equation}
and
\begin{equation}
    \mathcal{B}_s = \tilde{\mathcal{Q}}_s + \tilde{\mathcal{Q}}_s^T = \begin{bmatrix}
        \bm{0} & \bm{0} \\ \bm{0} & \mathcal{W}_{f}\text{diag} (\hat{\bm{n}}_s)
    \end{bmatrix},
\end{equation}
similar to the definitions of Chan~\cite{chan2019skew}.

{\color{black}In the case that the PDE is vector-valued with $N_s$ states, we assemble the residual per Eq.~\eqref{eq:STNSFR} for each state on each element.}

\begin{remark}
The discretization is similar to that of Friedrich \textit{et al.}~\cite{friedrich2019entropy} in that two-point fluxes are used in both spatial and temporal parts. 
Equation (\ref{eq:STNSFR}) is distinct from previous work in~\cite{friedrich2019entropy} as it allows the use of uncollocated nodes and applies FR.
\end{remark}

\subsection{Conservation and Stability of the ST-NSFR Scheme}

\begin{theorem}
    The ST-NSFR scheme is discretely globally conservative for any value of $c$.
\end{theorem}
\begin{proof}
For each equation of state, we left-multiply Eq.~\eqref{eq:STNSFR} by $\hat{\bm{1}}(\mathcal{M}+\mathcal{K})$ such that $\bm{\chi}(\bm{\bar{\xi}}_{soln})\hat{\bm{1}}^T = \bm{1}^T$,
\begin{equation*}
\begin{aligned}
    \hat{\bm{1}}(\mathcal{M}+\mathcal{K})(\mathcal{M}+\mathcal{K})^{-1} & 
    [\bm{\chi}(\bar{\bm{\xi}}_{q})^T \ \bm{\chi}(\bar{\bm{\xi}}_{f})^T ]
    \left[
    (\tilde{\mathcal{Q}} - \tilde{\mathcal{Q}}^T)_{s}
    \odot 
    \mathcal{F}^r_{s} \right] 
    \bm{1}^T \\
     &+ \hat{\bm{1}}(\mathcal{M}+\mathcal{K}){L}_{FR,s}\left( \bm{f}^{*,r}_s \right)  \\
     + \hat{\bm{1}}(\mathcal{M}+\mathcal{K})(\mathcal{M})^{-1}  &
     [\bm{\chi}(\bar{\bm{\xi}}_{q})^T \ \bm{\chi}(\bar{\bm{\xi}}_{f})^T]
    \left[
    (\tilde{\mathcal{Q}} - \tilde{\mathcal{Q}}^T)_{t}
    \odot 
    \mathcal{F}^r_{t} \right]
    \bm{1}^T \\ 
    &+ \hat{\bm{1}}(\mathcal{M}+\mathcal{K}){L}_{t} \left( \bm{f}^{*,r}_t
    \right)
    = \bm{0}^T.
\end{aligned}
\end{equation*}
\noindent which simplifies to the following after substitution of the lifting operators,
\begin{equation*}
\begin{aligned}
    \hat{\bm{1}}& 
    [\bm{\chi}(\bar{\bm{\xi}}_{q})^T \ \bm{\chi}(\bar{\bm{\xi}}_{f})^T ]
    \left[
    (\tilde{\mathcal{Q}} - \tilde{\mathcal{Q}}^T)_{s}
    \odot 
    \mathcal{F}_{s} \right] 
    \bm{1}^T
     + \hat{\bm{1}}(\mathcal{M}+\mathcal{K}){L}_{FR,s}\left( \bm{f}^*_s \right)  \\
     + \hat{\bm{1}}(\mathcal{M}+\mathcal{K})(\mathcal{M})^{-1}  &
     [\bm{\chi}(\bar{\bm{\xi}}_{q})^T \ \bm{\chi}(\bar{\bm{\xi}}_{f})^T]
    \left[
    (\tilde{\mathcal{Q}} - \tilde{\mathcal{Q}}^T)_{t}
    \odot 
    \mathcal{F}_{t} \right]
    \bm{1}^T \\ 
    &+ \hat{\bm{1}}(\mathcal{M}+\mathcal{K}){L}_{t} \left( \bm{f}^*_t
    \right)
    = \bm{0}^T.
\end{aligned}
\end{equation*}

Since both $(\tilde{\mathcal{Q}} - \tilde{\mathcal{Q}}^T)_{s}
    \odot 
    \mathcal{F}^r_{s}$ and $(\tilde{\mathcal{Q}} - \tilde{\mathcal{Q}}^T)_{t}
    \odot 
    \mathcal{F}^r_{t}$ are skew-symmetric, the temporal and spatial volume terms vanish. Alike Theorem \ref{thm:ESFR-conservation}, appropriate choice of numerical flux and boundary conditions -- i.e., periodic in space and Dirichlet in time -- leads to global conservation with any $c$.
\end{proof}
\subsection{The ST-NSFR Scheme is Entropy-preserving.}

We closely follow the structure of Theorems 1 and 2 in Friedrich \textit{et al.}~\cite{friedrich2019entropy}, extending the proofs to general node choices. We begin by focusing on the $c_{DG}$ case, followed by an analysis of the FR version.

\begin{theorem}
\label{thm:entropyeqn}
Let $\mathcal{F}_{s}$ be an entropy conservative spatial flux, $\mathcal{F}_{t}$ be an entropy conservative temporal flux, and the correction parameter be $c_{DG}$. Then the space-time NSFR system of equations~(\ref{eq:STNSFR}) 
{\color{black} left-multiplied by ${\bm{\hat v}} \mathcal{M}$ }over a single space-time element is equivalent to a form expressed in only surface fluxes,
\begin{equation}
\begin{aligned}
    \bm{1}\mathcal{W}_{f} \mathrm{diag}(\bm{\hat{n}}_t)\left[ \bm{\phi}^r(\bm{\tilde u}_f) - \bm{\tilde v}_f^T \bm{f}^{*,r}_t\right]
    +\bm{1} \mathcal{W}_f \mathrm{diag} (\bm{\hat{n}}_s) \left[ \bm{\psi}^r (\bm{\tilde u}_f) - \bm{\tilde v}^T \bm{f}^{*,r}_s\right] = 0.
    \label{eq:entropyeqn}
\end{aligned}
\end{equation}
which is an approximation of the statement of conservation of entropy~\eqref{eq:entropy_conservation_in_intform1}
\begin{equation}
\begin{aligned}    
    &\int_{\Omega_s^r} \left[ \left( \bm{\phi}^r (\bm{u}(\Pi\bm{v})) - (\Pi\bm{v})^T \bm{f}^{*,r}_t \right) \cdot \bm{\hat{n}}_t\right]_{-1}^1 \; \mathrm{d}\Omega_s^r \\
    &+\int_{-1}^{1} \int_{\Gamma_s^r}\left[ \bm{\psi}^r (\bm{u}(\Pi\bm{v})) - (\Pi\bm{v})^T \bm{f}^{*,r}_s\right] \cdot \bm{\hat{n}}_s \;\mathrm{d}\Gamma_s^r \mathrm{d}\tau= 0.
\end{aligned}
\end{equation}

\end{theorem}

\begin{proof}
    We begin by requiring that the discretization uses an entropy-conserving two-point flux for the volume term and numerical flux in the spatial dimension. {\color{black}We consider the case where the correction parameter is $c_{DG}=0$, corresponding to $\mathcal{K}=\bm{0}$}. The discretizaton in Eq.~(\ref{eq:STNSFR}) for the space-time element is left-multiplied by the modal coefficients of the entropy variable ${\bm{\hat v}} \mathcal{M}$, and yields 
    \begin{equation}
    \begin{aligned}        
    &\bm{\hat v}    
     [\bm{\chi}(\bar{\bm{\xi}}_{q})^T \ \bm{\chi}(\bar{\bm{\xi}}_{f})^T]
    \left[
    (\tilde{\mathcal{Q}} - \tilde{\mathcal{Q}}^T)_{t}
    \odot 
    \mathcal{F}^r_{t} \right]
    \bm{1}^T 
    +  \bm{\hat v} \bm{\chi}(\bar{\bm{\xi}}_{f})^T\mathcal{W}_f \ \text{diag} (\bm{\hat{n}}_t) \bm{f}^{*,r}_t \\
    + &\bm{\hat v}    
     [\bm{\chi}(\bar{\bm{\xi}}_{q})^T \ \bm{\chi}(\bar{\bm{\xi}}_{f})^T]
    \left[
    (\tilde{\mathcal{Q}} - \tilde{\mathcal{Q}}^T)_{s}
    \odot 
    \mathcal{F}^r_{s} \right]
    \bm{1}^T 
    +  \bm{\hat v} \bm{\chi}(\bar{\bm{\xi}}_{f})^T\mathcal{W}_f \ \text{diag} (\bm{\hat{n}}_s) \bm{f}^{*,r}_s
     = 0
    \end{aligned}
    \label{eq:single-elem-leftmult}
    \end{equation}

\noindent We first introduce ${\bm{\tilde v}}^T = \begin{bmatrix}
        \bm{\chi}(\bar{\bm{\xi}}_{q} )\\ \bm{\chi}(\bar{\bm{\xi}}_{f} )
\end{bmatrix} \bm{\hat v}^T$ on the solution and face nodes, for notational brevity. {\color{black}We will analyse temporal terms and spatial terms separately, which are found in the first line and second line of \eqref{eq:single-elem-leftmult} respectively.}
For the spatial terms, we will follow the approach of Chan~\cite{chan2018discretely,chan2019skew}, which is directly applicable to the underlying NSFR scheme~\cite{cicchino2025discretely}. Chan rewrites the volume hybrid term in flux-differencing form and sums over all quadrature nodes  $\color{black}N_{qf} = \{ N_{q},N_{f} \} $\cite[Eq.(71)]{chan2018discretely} to show
    \begin{equation}
        {\bm{\tilde v}}
        \left[
        (\tilde{\mathcal{Q}} - \tilde{\mathcal{Q}}^T)_{s}
        \odot 
        \mathcal{F}^r_{s} \right]
        \bm{1}^T = \sum_{i,j = 1}^{\color{black}N_{qf}} \tilde{Q}_{s,ij} \llbracket \bm{\tilde v} \rrbracket_{ij}^T \bm{f}^r_s({\bm{\tilde u}}_i,{\bm{\tilde u}}_j),
        \label{eq:entropy_eqn}
    \end{equation}
where $\bm{f}^r_s({\bm{\tilde u}}_i,{\bm{\tilde u}}_j)$ is the entropy-conserving flux evaluated with entropy-projected conservative variables, $\bm{\tilde u}=\bm{u}(\bm{\tilde v})$. Chan then applies the Tadmor shuffle condition to replace the change in entropy variables and the two‑point flux, which allows the entropy potential on the surface to be recovered through integration‑by‑parts~\cite[Eq.(72)–(73)]{chan2018discretely}. This substitution is feasible only because the two‑point flux had been constructed using the entropy‑projected variables{\color{black}, also highlighted in~\cite{parsaniEntropyStableStaggered2016}}. Since the Tadmor shuffle condition depends on the entropy variables in Eq.~\eqref{eq:entropy_vars}, the two‑point flux must satisfy
$$\llbracket \bm{\tilde v} \rrbracket_{ij}^T \bm{f}^r_s({\bm{\tilde u}}_i,{\bm{\tilde u}}_j) = \llbracket \bm{\psi}^r \rrbracket_{ij},$$
which requires mapping the conservative variables used in $\bm{f}_s$ from $\bm{v}$. Performing this substitution for the spatial terms in~\eqref{eq:entropy_eqn} yields 
\begin{equation}
    \begin{aligned}
    \bm{\tilde v}    
    \left[
    (\tilde{\mathcal{Q}} - \tilde{\mathcal{Q}}^T)_{s}
    \odot 
    \mathcal{F}^r_{s} \right]
    \bm{1}^T 
    + & \bm{\hat v} \bm{\chi}(\bar{\bm{\xi}}_{f})^T\mathcal{W}_f \ \text{diag} (\bm{\hat{n}}_s) \bm{f}^{*,r}_s \\
     =& -\bm{1} \mathcal{W}_f \text{diag} (\bm{\hat{n}}_s) \left[ \bm{\psi}^r (\bm{\tilde u}_f) - \bm{\tilde v}^T \bm{f}^{*,r}_s\right] \\
     \approx & - \int_{-1}^{1} \int_{\Gamma_s^r}\left[ \bm{\psi}^r (\bm{u}(\Pi\bm{v})) - (\Pi\bm{v})^T \bm{f}^{*,r}_s\right] \cdot \bm{\hat{n}}_s \;\mathrm{d}\Gamma_s^r \mathrm{d}\tau,
    \end{aligned}
\end{equation}
{\color{black}approximating conservation of entropy across the spatial boundaries $\Gamma^r_s$ on the reference time interval, written for generality with higher spatial dimensions.}

Having reduced the spatial terms to a surface integral, we now turn our attention to the temporal terms.  
Considering the volume temporal term, we note that it can be readily transposed,
\begin{equation}
        {\bm{\tilde v}}
        \left[
        (\tilde{\mathcal{Q}} - \tilde{\mathcal{Q}}^T)_{t}
        \odot 
        \mathcal{F}^r_{t} \right]
        \bm{1}^T 
        =
        {\bm{\tilde v}}
        \left[
        \tilde{\mathcal{Q}}_t
        \odot 
        \mathcal{F}^r_{t} \right]
        \bm{1}^T 
        -
        \bm{1}
        \left[
        \tilde{\mathcal{Q}}_t
        \odot 
        \mathcal{F}^r_{t} \right]
        {\bm{\tilde v}}^T 
\end{equation}

    We assume that the temporal two-point flux is an energy-stable flux $\bm{f}^r_t$. Noting that $\llbracket\bm{\tilde v}\rrbracket^T \bm{f}^r_t = \llbracket\bm{\phi}^r\rrbracket$ 
    {as shown in Section~\ref{sec:entropy-stability-CTS}}, we expand on all quadrature nodes 
    to demonstrate
    \begin{equation}
    \begin{aligned}
        {\bm{\tilde v}}
        \left[
        \tilde{\mathcal{Q}}_t
        \odot 
        \mathcal{F}^r_{t} \right]
        \bm{1}^T 
        -
        \bm{1}
        \left[
        \tilde{\mathcal{Q}}_t
        \odot 
        \mathcal{F}^r_{t} \right]
        {\bm{\tilde v}}^T  
        &= \sum_{i,j}^{N_q} \tilde{Q}_{t,ij} \llbracket \bm{\tilde v} \rrbracket_{ij}^T \bm{f}^r_t(\bm{\tilde u}_i,\bm{\tilde u}_j) = \sum_{i,j}^{N_q} \tilde{Q}_{t,ij} \llbracket{\bm{\phi}^r}\rrbracket_{ij} \\
        &= \bm{1}
        \tilde{\mathcal{Q}}_t
        {\bm{\phi}}^{r^T} 
        -
        {\bm{\phi}}^r
        \tilde{\mathcal{Q}}_t
        {\bm{1}}^T .
    \end{aligned}
    \end{equation}

    Because of the properties $\tilde{\mathcal{Q}}_t + \tilde{\mathcal{Q}}_t^T = \mathcal{B}_t$ and $\tilde{\mathcal{Q}}_t \bm{1}^T=\bm{0}^T$~\cite{chan2019skew},

    \begin{equation}
    \begin{aligned}
        {\bm{\phi}}^r
        \tilde{\mathcal{Q}}_t
        {\bm{1}}^T
        -
        \bm{1}
        \tilde{\mathcal{Q}}_t
        {\bm{\phi}}^{r^T} 
         = -\bm{1} (\tilde{\mathcal{B}}_t-\tilde{\mathcal{Q}}_t^T)\bm{\phi}^{r^T}
         = -\bm{1} \tilde{\mathcal{B}}_t\bm{\phi}^{r^T}.
    \end{aligned}
    \end{equation}
Thus, the volume term is expressed as a discrete integral of the entropy potential on the boundary of the element,
\begin{equation}
     {\bm{\tilde v}}
        \left[
        (\tilde{\mathcal{Q}} - \tilde{\mathcal{Q}}^T)_{t}
        \odot 
        \mathcal{F}^r_{t} \right]
        \bm{1}^T
        = -\bm{1} \tilde{\mathcal{B}}\bm{\phi}^{r^T}
        = -\bm{1}\mathcal{W}_{f} \ \text{diag}(\bm{\hat{n}}) \bm{\phi}^r(\bm{\tilde u}_f)
     \approx \int_{\Omega_s^r}  \left[ \bm{\phi}^r (\bm{u}(\Pi\bm{v}))  \cdot \bm{\hat{n}}_t \right]_{-1}^1\; \mathrm{d}\Omega_s^r
\end{equation}

We now {recombine the spatial and temporal parts to} express the final statement of the entropy equation within a space-time element,
\begin{equation}
    \bm{1}\mathcal{W}_{f} \text{diag}(\bm{\hat{n}}_t)\left[ \bm{\phi}^r(\bm{\tilde u}_f) - \bm{\tilde v}_f^T \bm{f}^{*,r}_t\right]
    +\bm{1} \mathcal{W}_f \text{diag} (\bm{\hat{n}}_s) \left[ \bm{\psi}^r (\bm{\tilde u}_f) - \bm{\tilde v}^T \bm{f}^{*,r}_s\right] = 0.
\end{equation}

\end{proof}

\begin{theorem}
\label{thm:entropypreservation}
Let $\mathcal{F}_{s}$ be an entropy conservative spatial flux and $\mathcal{F}_{t}$ be an entropy conservative temporal flux, both also used as the numerical flux, {\color{black}and let the correction parameter be $c_{DG}$}. Then consider the space-time entropy equation from Theorem~\ref{thm:entropyeqn} with Dirichlet boundary conditions in time and periodic boundary conditions in space. The space-time NSFR is then entropy preserving,
\begin{equation}
    s(\bm{\tilde u}(T)) -s(\bm{u}_0) 
    + \sum_{k=1}^{K_s} \bm{1}\mathcal{W}_{f} \mathrm{diag}(\bm{\hat{n}}^{-}_t) 
    \Bigl[ 
    \llbracket \bm{\phi}^r \rrbracket - \llbracket \bm{\tilde v}\rrbracket^T \bm{f}^{*,r}_t (\bm{\tilde u}^{-},\bm{\tilde u}^{+}) \Bigr] \Bigg|_{t=0} = 0,
    \label{eq:entropyconservation}
\end{equation}
{with initial numerical entropy calculated as 
\begin{equation*}
    s(\bm{u}_0) = \sum_{n=1}^{K_s} \bm{1} \mathcal{W}_{f,3} \mathrm{diag} (\bm{\hat{n}}^-_t) {J_{1D}}\bm{s}(\bm{u}_0)
\end{equation*}
and likewise for final numerical entropy.}
\end{theorem}

\begin{proof}
{\color{black}First, we express the interpolation of the temporal flux to the boundaries of a single space-time element as
\begin{equation}
\begin{aligned}
        \bm{\tilde v}[\mathcal{B}_t \odot \mathcal{F}^r_{t}] \bm{1}^T
        &= \bm{\tilde{v}}\left(\begin{bmatrix}
            0 & 0 \\
            0 & \mathcal{W}_f \ \text{diag}(\bm{\hat{n}}_t)
        \end{bmatrix} \odot \mathcal{F}^r_{t}\right)
        \bm{1}^T \\
        &= \bm{\hat v}\bm{\chi}(\bar{\bm{\xi}}_{f})^T\mathcal{W}_f \ \text{diag} (\bm{\hat{n}}_t) \bm{f}^r_t(\tilde{\bm{u}}).
    \end{aligned}
\end{equation}
We add and subtract the temporal flux on the boundaries to Eq.~\eqref{eq:entropyeqn}. The temporal terms on a single space-time element are therefore rearranged to reveal,
\begin{equation}
\begin{aligned}    
    &-\bm{1}\mathcal{W}_{f} \text{diag}(\bm{\hat{n}}_t)\left[ \bm{\phi}^r(\bm{\tilde u}_f) - \bm{\tilde v}_f^T \bm{f}^r_t(\bm{\tilde u}_f)\right] + \bm{1}\mathcal{W}_{f} \ \text{diag}(\bm{\hat{n}}_t) \bm{\tilde v}_f^T \left[ \bm{f}^{*,r}_t -  \bm{f}^r_t(\bm{\tilde u}_f)\right] \\
    &-\bm{1} \mathcal{W}_f \text{diag} (\bm{\hat{n}}_s) \left[ \bm{\psi}^r (\bm{\tilde u}_f) - \bm{\tilde v}^T \bm{f}^{*,r}_s\right] = 0.
\end{aligned}
\end{equation}}
The proof for global entropy conservation then proceeds with a summation over all $K_s$-space-time elements and $K_t$-timeslabs {\color{black}numbered per Section \ref{sec:elem-numbering}}. From Eq.~\eqref{eq:entropyeqn}, we have
    \begin{equation}
    \begin{aligned}        
        \sum_{n=1}^{K_T} \sum_{k=1}^{K_s}     & \Bigl[ -\bm{1}\mathcal{W}_{f} \text{diag}(\bm{\hat{n}}_t)\left[ \bm{\phi}^r(\bm{\tilde u}_f) - \bm{\tilde v}_f^T \bm{f}^r_t(\bm{\tilde u}_f)\right] + \bm{1}\mathcal{W}_{f} \ \text{diag}(\bm{\hat{n}}_t) \bm{\tilde v}_f^T \left[ \bm{f}^{*,r}_t -  \bm{f}^r_t (\bm{\tilde u}_f)\right] \\
    &-\bm{1} \mathcal{W}_f \text{diag} (\bm{\hat{n}}_s) \left[ \bm{\psi}^r (\bm{\tilde u}_f) - \bm{\tilde v}^T \bm{f}^{*,r}_s\right] \Bigr] = 0.
    \end{aligned}
    \label{eq:sumofallspacetimeelements}
    \end{equation}
    
\noindent We now impose fully periodic boundary conditions in space, expand the term having spatial interface contributions across each timeslab, and substitute $\bm{\hat{n}}^+_s = -\bm{\hat{n}}^-_s$. We note that summation indices are omitted for notational brevity and the `$+,-$', notation is employed instead to denote exterior and interior entropy-projected conservative states across the spatial interface, 
    \begin{equation}
    \begin{aligned}
    \sum_{n=1}^{K_T} \sum_{k=1}^{K_s} -\bm{1} \mathcal{W}_f \text{diag} (\bm{\hat{n}}_s) &[ \bm{\psi}^r (\bm{\tilde u}_f) - \bm{\tilde v}^T \bm{f}^{*,2}_s ] = \sum_{n=1}^{K_T} \Bigl[ \cdots
     +\bm{1} \mathcal{W}_f \text{diag} (\bm{\hat{n}}^-_s) \left[\bm{\psi}^r (\bm{\tilde u}^+_f) -\bm{\psi}^r (\bm{\tilde u}^-_f)\right] \\
    & - \bm{1}\mathcal{W}_f \text{diag} (\bm{\hat{n}}^-_s) \left( (\bm{\tilde v}^+_f )^T - (\bm{\tilde v}^-_f )^T \right) \bm{f}^r_s (\bm{\tilde u}^+_f, \bm{\tilde u}^-_f) + \cdots \Bigr].
    \end{aligned}
    \end{equation}
    Through the application of Tadmor's shuffle condition {\color{black} across all spatial interfaces, including the periodic boundaries}, $$\left( (\bm{\tilde v}^+_f )^T - (\bm{\tilde v}^-_f )^T \right) \bm{f}_s (\bm{\tilde u}^+_f, \bm{\tilde u}^-_f) = \bm{\psi}^r (\bm{\tilde u}^+_f) -\bm{\psi}^r (\bm{\tilde u}^-_f),$$ the summation of the spatial interface contributions across all space-time elements and timeslabs cancels {all interface terms} {\color{black} and only temporal terms remain in Eq. \eqref{eq:sumofallspacetimeelements}.} 

    We now turn our attention to the first boundary term across the temporal interfaces and impose temporal Dirichlet initial boundary conditions. From Eq.~\eqref{eq:sumofallspacetimeelements}, we unroll and assemble all temporal face terms, substitute $\bm{\hat{n}}^+_t = -\bm{\hat{n}}^-_t$, and now denote the `$+,-$', notation as exterior and interior entropy-projected conservative states across the temporal interfaces,
    \begin{equation}
        \begin{aligned}
            \sum_{n=1}^{K_T} \sum_{k=1}^{K_s} -\bm{1}\mathcal{W}_{f} \text{diag}(\bm{\hat{n}}_t) & \left[ \bm{\phi}^r(\bm{\tilde u}_f) - \bm{\tilde v}_f^T \bm{f}^r_t(\bm{\tilde u}_f)\right] \\= &\sum_{n=1}^{K_s} \Bigl[ \cdots
     +\bm{1} \mathcal{W}_f \text{diag} (\bm{\hat{n}}^-_t) \left[(\bm{\tilde v}^{-}_f )^T\bm{f}^r_t(\bm{\tilde u}^{-}_f) -\bm{\phi}^r(\bm{\tilde u}^-_f)\right]\Bigg|_{t=T} \\
     & \cdots - \bm{1}\mathcal{W}_f \text{diag} (\bm{\hat{n}}^-_t) \left[(\bm{\tilde v}^{+}_f )^T\bm{f}^r_t(\bm{\tilde u}^{+}_f) -(\bm{\tilde v}^{-}_f )^T\bm{f}^r_t(\bm{\tilde u}^{-}_f) +\bm{\phi}^r(\bm{\tilde u}^+_f) -\bm{\phi}^r(\bm{\tilde u}^-_f)\right] \\
     &\cdots -\bm{1} \mathcal{W}_f \text{diag} (\bm{\hat{n}}^-_t) \left[(\bm{\tilde v}^{+}_f )^T\bm{f}^r_t(\bm{\tilde u}^{+}_f) -\bm{\phi}^r(\bm{\tilde u}^+_f)\right]\Bigg|_{t=0} \Bigr].
        \end{aligned}
    \end{equation}
    {\color{black}The terms across interior interfaces cancel, leaving only terms at the $t=0$ and $t=T$ boundaries.}
    Through the definition of the entropy/entropy-flux potential pair $\phi = \bm{v}^T\bm{u} -s$, we can express the above expansion as,
    \begin{equation}
        \begin{aligned}
            \sum_{n=1}^{K_T} \sum_{k=1}^{K_s} -\bm{1}\mathcal{W}_{f} \text{diag}(\bm{\hat{n}}_t) & \left[ \bm{\phi}^r(\bm{\tilde u}_f) - \bm{\tilde v}_f^T \bm{f}^r_t(\bm{\tilde u}_f)\right] \\ &= \sum_{n=1}^{K_s} \Bigl[ \bm{1} \mathcal{W}_f \text{diag} (\bm{\hat{n}}^-_t) \bm{s}(\bm{\tilde u}^{-}_f)\Big|_{t=T} -\bm{1} \mathcal{W}_f \text{diag} (\bm{\hat{n}}^-_t) s(\bm{\tilde u}^{+}_f)\Big|_{t=0} \Bigr] \\
            & = \bm{s}(\bm{\tilde u}(T)) - s(\bm{\tilde u}(0)) 
        \end{aligned}
    \end{equation}
We now turn our attention to the second temporal boundary term, and as such, with the substitution of $\bm{\hat{n}}^+_t = -\bm{\hat{n}}^-_t$, the second term from Eq.~\eqref{eq:sumofallspacetimeelements} can be written as 
    \begin{equation}
        \begin{aligned}
            \sum_{n=1}^{K_T} \sum_{k=1}^{K_s}     &\;  \bm{1}\mathcal{W}_{f} \ \text{diag}(\bm{\hat{n}}_t) \bm{\tilde v}_f^T \left[ \bm{f}^{*,r}_t -  \bm{f}^r_t (\bm{\tilde u}_f)\right] \\=& \sum_{n=1}^{K_s} \Bigl[ \cdots
     +\bm{1} \mathcal{W}_f \text{diag} (\bm{\hat{n}}^-_t) (\bm{\tilde v}^{-}_f )^T\left[\bm{f}^{*,r}_t(\bm{\tilde u}^{-}_f,\bm{\tilde u}^{+}_f) -\bm{f}^r_t(\bm{\tilde u}^-_f)\right]\Bigg|_{t=T} \\
     & \cdots - \bm{1}\mathcal{W}_f \text{diag} (\bm{\hat{n}}^-_t) \left[\left((\bm{\tilde v}^{+}_f )^T -(\bm{\tilde v}^{-}_f )^T\right)\bm{f}^{*,r}_t(\bm{\tilde u}^{-}_f,\bm{\tilde u}^{+}_f) -\left( (\bm{\tilde v}^{+}_f )^T\bm{f}^r_t(\bm{\tilde u}^{+}_f) -(\bm{\tilde v}^{-}_f )^T\bm{f}^r_t(\bm{\tilde u}^{-}_f) \right) \right] \\
     &\cdots -\bm{1} \mathcal{W}_f \text{diag} (\bm{\hat{n}}^-_t) (\bm{\tilde v}^{+}_f )^T \left[\bm{f}^{*,r}_t(\bm{\tilde u}^{-}_f,\bm{\tilde u}^{+}_f) -\bm{f}^r_t(\bm{\tilde u}^+_f)\right]\Bigg|_{t=0} \Bigr]
        \end{aligned}
        \label{eq:expandedtemporalterm}
    \end{equation}    
At this stage, we introduce Dirichlet boundary conditions for the initial condition at $t=0$, pure upwinding at the final time, $t=T$, as it is the natural choice, and ensure that the temporal numerical flux, $\bm{f}^*_t$ satisfies the temporal equivalent of the Tadmor shuffle condition of Eq.~\eqref{eq:Tadmor-shuffle}. These conditions can be written as
\begin{equation}
    \begin{aligned}
        &\bm{f}^{*,r}_t(\bm{\tilde u}^{-}_f,\bm{\tilde u}^{+}_f)\Big|_{t=T} = \bm{\tilde u}^{-}_f,  \qquad \bm{f}^{*,r}_t(\bm{\tilde u}^{-}_f,\bm{\tilde u}^{+}_f)\Big|_{t=0} = \bm{u}_0 \\
        & \left((\bm{\tilde v}^{+}_f )^T -(\bm{\tilde v}^{-}_f )^T\right)\bm{f}^{*,r}_t(\bm{\tilde u}^{-}_f,\bm{\tilde u}^{+}_f) = \bm{\phi}^r(\bm{\tilde u}^{+}_f) -\bm{\phi}^r(\bm{\tilde u}^{-}_f) \quad \text{for all internal temporal interfaces}.
    \end{aligned}
    \label{eq:temporalentropycons_bc}
\end{equation}   
At the temporal interface $t=T$, the boundary integral vanishes due to the pure upwinding condition. We can now assemble the final form of Eq.~\eqref{eq:sumofallspacetimeelements} {\color{black}by combining the reduced forms of the two temporal terms},
    \begin{equation}
    \begin{aligned}        
        \sum_{n=1}^{K_T} \sum_{k=1}^{K_s}     & \Bigl[ -\bm{1}\mathcal{W}_{f} \text{diag}(\bm{\hat{n}}_t)\left[ \bm{\phi}^r(\bm{\tilde u}_f) - \bm{\tilde v}_f^T \bm{f}^r_t(\bm{\tilde u}_f)\right] + \bm{1}\mathcal{W}_{f} \ \text{diag}(\bm{\hat{n}}_t) \bm{\tilde v}_f^T \left[ \bm{f}^{*,r}_t -  \bm{f}^r_t (\bm{\tilde u}_f)\right] \\
    &-\bm{1} \mathcal{W}_f \text{diag} (\bm{\hat{n}}_s) \left[ \bm{\psi}^r (\bm{\tilde u}_f) - \bm{\tilde v}^T \bm{f}^{*,r}_s\right] \Bigr] \\ &= s(\bm{\tilde u}(T)) - s(\bm{\tilde u}(0)) + \sum_{n=1}^{K_s} \Bigl[
    \bm{1} \mathcal{W}_f \text{diag} (\bm{\hat{n}}^-_t) (\bm{\tilde v}^{+}_f )^T \left[\bm{\tilde u}^{+}_f -\bm{\tilde u}^-_f\right]\Bigg|_{t=0} \Bigr] = 0.
    \end{aligned}
    \end{equation}
    
    {\color{black}We seek a form which relates the final entropy $s(\bm{\tilde u}(T))$ to the initial condition applied at $t=0$, $s(\bm{u}_0)$.} To arrive at such a statement, we add and subtract the initial entropy, $s(\bm{u}_0)$, and the product $(\bm{\tilde v}^-_f)^T\bm{\tilde u}^-_f$ for the exterior values for the first temporal slab at the $t=0$ interface,
    \begin{equation}
    \begin{aligned}    
    s(\bm{\tilde u}(T)) - s(\bm{u}_0) &+ \sum_{n=1}^{K_s} 
    \bm{1} \mathcal{W}_f \text{diag} (\bm{\hat{n}}^-_t) \Bigl[(\bm{\tilde v}^{+}_f )^T \bm{\tilde u}^{+}_f 
    -(\bm{\tilde v}^{-}_f )^T \bm{\tilde u}^{-}_f
    \\ &
    -\left( s(\bm{\tilde u}^{+}_f) -s(\bm{u}_0)\right)
    -\left((\bm{\tilde v}^{+}_f )^T -(\bm{\tilde v}^{-}_f )^T\right)\bm{\tilde u}^{-}_f
     \Bigr]\Bigg|_{t=0} = 0.
    \end{aligned}
    \end{equation}
Through another application of the entropy/entropy-flux potential pair, we arrive at the final form,
\begin{equation}
    s(\bm{\tilde u}(T)) -s(\bm{u}_0) 
    + \sum_{k=1}^{K_s} \bm{1}\mathcal{W}_{f} \textup{diag}(\bm{\hat{n}}^{-}_t) 
    \Bigl[ \llbracket \bm{\phi}^r \rrbracket - \llbracket \bm{\tilde v}_f\rrbracket^T \bm{f}^{*,r}_t (\bm{\tilde u}^{-}_f,\bm{\tilde u}^{+}_f) \Bigr] \Bigg|_{t=0} = 0.
\end{equation}
From~\cite{friedrich2019entropy}, the final term above is designated as the projection error due to the enforcement of the Dirichlet boundary condition. 
\end{proof}

\begin{remark}
    The projection error is evaluated by taking $\bm{f}^{*,r}_t (\bm{\tilde u}^{-}_f,\bm{\tilde u}^{+}_f)\Big|_{t=0} = \bm{u}_0$, where $\bm{u}_0$ is evaluated at the temporal face quadrature points; while the exterior entropy variable, $\bm{\tilde v}^{+}_f$ at $t=0$ is evaluated using the same $\bm{u}_0$ and $\bm{\phi}(\bm{\tilde u}^{+}_f) $. 
\end{remark}

\begin{theorem}
\label{thm:entropystability}
Let $\mathcal{F}_{s}$ be an entropy conservative spatial flux and $\mathcal{F}_{t}$ be an entropy conservative temporal flux, while a pure upwinding is considered for the temporal numerical flux, $\bm{f}^*_t$ and an entropy-stable numerical flux at the spatial boundaries, $\bm{f}^*_s$. Then consider the space-time entropy equation from Theorem~\ref{thm:entropyeqn} with Dirichlet boundary conditions in time and periodic boundary conditions in space, then the space-time NSFR with $c_{DG}$ is entropy-stable,
\begin{equation}
s(\bm{\tilde u}(T)) \le s(\bm{u}_0)    
\label{eq:statement-of-entropy-stability}
\end{equation}
\end{theorem}
\begin{proof}
With an appropriate choice of an entropy-stable spatial numerical flux, the baseline NSFR spatial discretization has been shown in~\cite{cicchino2025discretely} to be entropy-stable. As for the temporal terms, entropy stability hinges on ensuring that the second term in Eq.~\eqref{eq:sumofallspacetimeelements} remains positive. It remains to be shown that contributions to the internal temporal interfaces, as shown in Eq.~\eqref{eq:temporalentropycons_bc}, as well as the projection term at $t=0$ are non-negative. 

Friedrich \textit{et al.}~\cite{friedrich2019entropy} demonstrated that using pure upwinding in the temporal direction, specified as 
\color{black}
\begin{equation}
    \begin{aligned}
        &\bm{f}^{*,r}_t(\bm{\tilde u}^{-}_f,\bm{\tilde u}^{+}_f)\Big|_{t=T} = \bm{\tilde u}^{-}_f,  \qquad \bm{f}^{*,r}_t(\bm{\tilde u}^{-}_f,\bm{\tilde u}^{+}_f)\Big|_{t=0} = \bm{u}_0 \\
        & \bm{f}^{*,r}_t(\bm{\tilde u}^+_f,\bm{\tilde u}^-_f) = \bm{\tilde u}^-_f \quad \text{for all internal temporal interfaces};
    \end{aligned}
    \label{eq:temporalentropystab_bc}
\end{equation}   
the method fulfills
\begin{equation}
    \llbracket \bm{\phi}^r \rrbracket \ge \llbracket \bm{\tilde v} \rrbracket^T\bm{f}^{*,r}_t.
\label{eq:temporalTadmorCondition}
\end{equation}
Hence, both contributions from Eq.~\eqref{eq:temporalentropycons_bc} and Eq.~\eqref{eq:entropyconservation} remain non-negative. {\color{black} The node-agnostic formulation does not impact any conclusions therein}.
\end{proof}

We now extend the scheme and establish entropy stability for flux reconstruction in the case where the parameter $c$ is nonzero, and consequently, the matrix $\mathcal{K}$ is also nonzero. The discretization in Eq.~(\ref{eq:STNSFR}) for the space-time element is then left-multiplied by the product of the modal coefficients of the entropy variable and the FR modified mass matrix ${\bm{\hat v}}(\mathcal{M +K})$. Unlike the spatial terms, where the FR mass-matrix cancels the inverse that appears in Eq.~(\ref{eq:STNSFR}), the temporal term yields a residual term pre-multiplied by $\mathcal{KM}^{-1}$, which prevents a formal proof for entropy preservation and stability from being established. From the following analysis, we show that the additional term scales with the FR parameter $c$, and can be shown to be entropy-stable. We confirm the findings through numerical experiments.

The starting point of our analysis is the observation that all the temporal terms that contribute toward entropy preservation in Theorem~\ref{thm:entropypreservation} and entropy stability in Theorem~\ref{thm:entropystability} hinge on satisfying Eq.~\eqref{eq:temporalTadmorCondition}. 
That is, the right-hand-side  of Eq.~\eqref{eq:temporalTadmorCondition}  is the difference of two inner products in $L_2$. 
These terms arose when Eq.~\eqref{eq:STNSFR} was premultiplied by $\bm{\hat v} \mathcal{M}$. As such one of these terms can be expressed in inner-product form as,
\begin{equation}
\bm{\hat v} \mathcal{M} L_t\bm{f}^{*,r}_t \approx  \left< (\Pi \bm{v})^T, L_t\bm{f}^{*,r}_t \right>_{L^2(\Omega_s\times(0,T))} 
\end{equation}
However, in the case of FR, the inner-product appears as 
\begin{equation}
\bm{\hat v} (\mathcal{M+K}) L_t\bm{f}^{*,r}_t \approx  \left< (\Pi \bm{v})^T, L_t\bm{f}^{*,r}_t \right>_{H^k(\Omega_s\times(0,T))} 
\end{equation}

\begin{remark}    
In general, suppose the $L^2$ inner product is positive:
$$
\langle u,v\rangle_{L^2(\Omega_s)}
=
\int_{\Omega_s} u(x)v(x)\, dx
> 0,
$$
then nothing can be said about the sign of
$$
\langle u,v\rangle_{H^k(\Omega_s)}
=
\sum_{|\alpha|\le k}
\langle D^\alpha u, D^\alpha v\rangle_{L^2(\Omega_s)},
$$
and hence the derivative terms may change the sign. The same is true if the $L^2$ inner product is negative. Hence, to show that Eq.~\eqref{eq:temporalTadmorCondition} holds for FR, we must demonstrate that the sign of the existing inner products in $L^2$ remains as is for sufficiently small values of the flux-reconstruction parameter $c$. We establish this property in the following Theorem~\ref{Thm:innerproduct_in_Sobolev}{\color{black}, where we provide a proof for the case where the $L^2$ inner product is positive.}
\end{remark}
\begin{theorem}
\label{Thm:innerproduct_in_Sobolev}
There is a range of values for the flux-reconstruction parameter $c$ such that the inner product in the broken Sobolev space satisfies the following condition
\begin{equation}
\langle u,v\rangle_{H^k} \ge
\langle u,v\rangle_{L^2}
-
c\|D^p u\|_{L^2}\,\|D^p v\|_{L^2} \ge 0.   
\end{equation}

\end{theorem}
\begin{proof}

The $H^k$ inner product, using the standard equivalent norm, can be written as  
$$
\langle u,v\rangle_{H^k}
=
\sum_{|\alpha|\le k}
\langle D^\alpha u, D^\alpha v\rangle_{L^2}
=
\langle u,v\rangle_{L^2}
+
\sum_{1\le|\alpha|\le k}
\langle D^\alpha u, D^\alpha v\rangle_{L^2}.
$$
So the {\color{black} positive} $L^2$ term is only one part of the $H^k$ inner product, while the higher-derivative terms may change the sign.
Let
\begin{equation}
    S :=
\sum_{1\le|\alpha|\le k}
\langle D^\alpha u, D^\alpha v\rangle_{L^2}.
\label{eq:S-inner-product}
\end{equation}
Then
$$
|S|
\le
\sum_{1\le|\alpha|\le k}
\left| \langle D^\alpha u, D^\alpha v\rangle_{L^2} \right|.
$$
Since the sign of $S$ is unknown, we can then establish a  lower bound for $S$; where
$$
S
\ge
-
\sum_{1\le|\alpha|\le k}
\left| \langle D^\alpha u, D^\alpha v\rangle_{L^2} \right|.
$$
We now add $\langle u,v\rangle_{L^2}$ to both sides of the equation to establish
$$
\langle u,v\rangle_{H^k}
= 
\langle u,v\rangle_{L^2} + S
\ge
\langle u,v\rangle_{L^2}
-
\sum_{1\le|\alpha|\le k}
\left| \langle D^\alpha u, D^\alpha v\rangle_{L^2} \right|.
$$
We can establish a strict lower bound through the application of Cauchy--Schwarz; where
$$
\left| \langle D^\alpha u, D^\alpha v\rangle_{L^2} \right|
\le
\|D^\alpha u\|_{L^2}\,\|D^\alpha v\|_{L^2}.
$$
Hence, we can then establish the following inequality, 
\begin{equation}
\langle u,v\rangle_{H^k} \ge
\langle u,v\rangle_{L^2}
-
\sum_{1\le|\alpha|\le k}
\|D^\alpha u\|_{L^2}\,\|D^\alpha v\|_{L^2}    
\label{eq:generalSobolevinnerproduct}
\end{equation}
If the positive $L^2$ contribution dominates the derivative cross-terms, e.g. if
$$
\sum_{1\le|\alpha|\le k}
\|D^\alpha u\|_{L^2}\,\|D^\alpha v\|_{L^2} < \langle u,v\rangle_{L^2},
$$
then $\langle u,v\rangle_{H^k} > 0$. Equivalently, if the higher-derivative correlations are small relative to the $L^2$ correlation, positivity is preserved. This provides a sufficient condition for $\langle u,v\rangle_{H^k} > 0$. In FR, a parameter $c$ scales the $p$-th derivative term and therefore the inner-product can then be rewritten using Eq.~\eqref{eq:generalSobolevinnerproduct} as
\begin{equation}
\langle u,v\rangle_{H^k_c} \ge
\langle u,v\rangle_{L^2}
-
c\|D^p u\|_{L^2}\,\|D^p v\|_{L^2} \ge 0.   
\end{equation}
Therefore, there is a range of values of $c$ for which the broken Sobolev norm remains positive.
\end{proof}

\begin{remark}
    Since the sign of the term $S$ {\color{black}(Eq.~\eqref{eq:S-inner-product})} of the $p$-the derivative term above could be either negative or positive, we can extend Theorem~\ref{Thm:innerproduct_in_Sobolev} and state that 
\begin{equation}
    \langle u,v\rangle_{L^2}
- c\|D^p u\|_{L^2}\,\|D^p v\|_{L^2} \le     \langle u,v\rangle_{H^k_c} \le
\langle u,v\rangle_{L^2}
+ c\|D^p u\|_{L^2}\,\|D^p v\|_{L^2}.
\label{eq:rangeforNSFRStability}
\end{equation}
The broken Sobolev inner product differs from the $L^2$ inner product by at most $S$ so the derivative terms can shift the value by at most $S$ in either direction. Eq.~\eqref{eq:rangeforNSFRStability} holds regardless of whether $\langle u,v\rangle_{L^2}$ is positive or negative. Hence, we claim that the flux-reconstruction variant of the space-time entropy-stable approach satisfies Eq.~\eqref{eq:temporalTadmorCondition} and remains entropy-stable for sufficiently small $c$.
\end{remark}

\section{Space-Time Nonlinearly-stable Flux Reconstruction for Scalar Conservation Laws}
\label{sec:Burgers-results}

We consider the analysis of Section~\ref{sec:ST-NSFR} in the context of scalar-valued conservation laws,
\begin{equation}
    \begin{cases}
        \frac{\partial u}{\partial t} +  \frac{\partial f_s(u)}{\partial x} =0 \ \ \ 
        \forall (x, t) \in ([x_L, x_R]\times[0, T])\\
        u(x_L,t) = u(x_R,t)\\
        u(x, 0) = u_0(x)
    \end{cases}
    \label{eq:scalar-nonlin-cons-law}
\end{equation}
for which we can define entropy as the norm $\frac{1}{2}||\bm{{u}}||_{{H^k_c}}$. The entropy potential is $\phi = u^2/2$ and the entropy variable is the state, $v=u$.

\begin{remark}
\label{rmk:burgers-preservation}
   Scalar conservation laws alike Eq.~\eqref{eq:scalar-nonlin-cons-law} can be shown to be entropy-stable in the broken Sobolev norm according to the analysis of Theorems~\ref{thm:entropyeqn}-\ref{Thm:innerproduct_in_Sobolev}. The inner-product numerical entropy results in an inner-product form of Eq.~\eqref{eq:entropyconservation},

    \begin{equation}
    \begin{aligned}
        \frac{1}{2} \sum_{k=1}^{K_s}||\bm{{u}}_f (T) &||_{{H^k_c}}  
        -\frac{1}{2} \sum_{k=1}^{K_s}||\bm{u}_0||_{{H^k_c}} \\
    &+ \sum_{k=1}^{K_s} \left[ \frac{1}{2}\llbracket ||{\bm{u}}_f||_{{H^k_c}}\rrbracket  
    -\langle \llbracket{\bm{u}}_f \rrbracket , \bm{u}_0 \rangle_{{H^k_c}}
    \right]\Bigg|_{t=0} = 0,
    \end{aligned}
    \label{eq:burgers-entropy-preservation}
    \end{equation}    
    enabling entropy preservation in the broken Sobolev norm. We note that the entropy projection of Eq.~\eqref{eq:entropy-projection} is not necessary in the scalar case, as the entropy and conservative variables are the same.
\end{remark}

{\color{black}The inner products appearing in Remark~\ref{rmk:burgers-preservation} are calculated as
\begin{equation}
    \langle \bm{u},\bm{u} \rangle_{{H^k_c}}
    =
    (\bm{\chi}(\bm{\bar{\xi}}_{f,3})\hat{\bm{u}})
    (\mathcal{M}_{1D} + \mathcal{K}_{1D})(\bm{\chi}(\bm{\bar{\xi}}_{f,3})\hat{\bm{u}})^T,
    \label{eq:energy-calculation-example}
\end{equation}
where the one-dimensional $\mathcal{M}_{1D}$ and $\mathcal{K}_{1D}$ are defined alike Section~\ref{sec:ST-ESFR} but using the flux basis $\bm{\phi}$ rather than the solution basis. Equation \eqref{eq:energy-calculation-example} demonstrates a calculation on face $3$ of the space-time reference element per Fig.~\ref{fig:reference_elem}; a similar formulation is used on face $4$.}

We use the Burgers' equation for numerical experiments, which is of the form of Eq.~\eqref{eq:scalar-nonlin-cons-law} with $f_s=u^2/2$. We use $x_L = 0, x_R=2$ and $T=2$ to define a computational domain. To test convergence behaviour, we add a manufactured source
\begin{equation}
    q(x,t) = \pi  \sin(\pi (x-t))\big(1-\cos(\pi (x-t))\big),
    \label{eq:burgers-source}
\end{equation}
which we calculate at the solution nodes and project to the solution basis.
The source term as in Eq. (\ref{eq:burgers-source}) results in an analytical solution at any time,
\begin{equation}
    u_{\text{exact}}(x,t) = \cos(\pi (x-t)),
\end{equation}
which is also used to set the Dirichlet boundary condition at $t=0$. 
The same manufactured solution is used in \cite{cicchino2022nonlinearly}.
We use the entropy-stable two-point temporal state
\begin{equation}
    f_t(u_i,u_j) = \frac{1}{2}(u_i+u_j)
    \label{eq:burgers-temporal-state}
\end{equation}
which we derive in \ref{app:burgers_num_state}. The temporal numerical flux is purely upwind. We use the entropy-conserving spatial two-point flux,
\begin{equation}
    f(u_i,u_j) = \frac{1}{6}(u_i^2 + u_i u_j + u_j^2).
    \label{eq:burgers-spatial-num-flux}
\end{equation}
For convergence tests, we add a local Lax-Friedrichs dissipation term to the spatial flux to obtain an entropy-stable discretization. We solve the implicit system as described in Section~\ref{sec:lin-adv-results}. $L_2$ error is calculated using overintegration by 10 per dimension. The results of the convergence studies using collocated and uncollocated nodes are included in Tables \ref{tab:con-burgers-GLLGL} and \ref{tab:con-burgers-GLGL}. We observe optimal $p+1$ convergence for all $p=4$ configurations. The $p=3$ results lose optimal convergence before $c_{Hu}$, but maintain $p+1$ for the smaller $c=1E-5$, verifying that we reach optimal convergence for nonzero $c$ values. With the exception of $c_{Hu}$, error levels are higher in Table~\ref{tab:con-burgers-GLLGL} compared to Table~\ref{tab:con-burgers-GLGL}, reflecting the differing strength of the integration rules.

\begin{table}[h!]
\centering
\caption{Convergence of the Burgers' ST-NSFR method using GLL solution nodes and GL flux nodes.}
\label{tab:con-burgers-GLLGL}
\begin{tabular}{|c|c|ll|ll|ll|}
\hline
\multirow{2}{*}{Convergence test} & \multirow{2}{*}{$N$} & \multicolumn{2}{c|}{$c_{DG}$} & \multicolumn{2}{c|}{$c=1E-5$} & \multicolumn{2}{c|}{$c_{Hu}$} \\
                                  &                      & $L_2$ error       & rate      & $L_2$ error       & rate      & $L_2$ error       & rate      \\ \hline
                                  & 2                    & 2.31E-01          & --        & 2.30E-01          & --        & 1.94E-01          & --        \\
                                  & 4                    & 1.16E-02          & 4.32      & 1.15E-02          & 4.32      & 2.91E-02          & 2.74      \\
                                  & 8                    & 4.71E-04          & 4.62      & 4.78E-04          & 4.59      & 3.30E-03          & 3.14      \\
$p=3$                             & 16                   & 1.94E-05          & 4.60      & 2.06E-05          & 4.54      & 2.84E-04          & 3.54      \\
                                  & 32                   & 9.68E-07          & 4.32      & 1.04E-06          & 4.31      & 2.19E-05          & 3.70      \\
                                  & 64                   & 5.68E-08          & 4.09      & 6.29E-08          & 4.05      & 1.79E-06          & 3.61      \\
                                  & 128                  & 3.49E-09          & 4.02      & 4.35E-09          & 3.85      & 1.31E-07          & 3.77      \\ \hline
                                  & 2                    & 3.78E-02          & --        & 5.93E-02          & --        & 6.63E-02          & --        \\
                                  & 4                    & 1.49E-03          & 4.66      & 4.38E-03          & 3.76      & 5.20E-03          & 3.67      \\
                                  & 8                    & 4.12E-05          & 5.18      & 1.50E-04          & 4.87      & 1.80E-04          & 4.85      \\
$p=4$                             & 16                   & 9.06E-07          & 5.51      & 4.42E-06          & 5.09      & 5.37E-06          & 5.06      \\
                                  & 32                   & 1.91E-08          & 5.57      & 1.26E-07          & 5.13      & 1.54E-07          & 5.12      \\
                                  & 64                   & 4.71E-10          & 5.34      & 3.79E-09          & 5.05      & 4.73E-09          & 5.03      \\
                                  & 128                  & 1.33E-11          & 5.15      & 1.18E-10          & 5.00      & 1.48E-10          & 5.00      \\ \hline
\end{tabular}
\end{table}

\begin{table}[h!]
\centering
\caption{Convergence of the Burgers' ST-NSFR method using collocated GL solution nodes and flux nodes.}
\label{tab:con-burgers-GLGL}
\begin{tabular}{|c|c|ll|ll|ll|}
\hline
\multirow{2}{*}{Convergence test} & \multirow{2}{*}{$N$} & \multicolumn{2}{c|}{$c_{DG}$} & \multicolumn{2}{c|}{$c=1E-5$} & \multicolumn{2}{c|}{$c_{Hu}$} \\
                                  &                      & $L_2$ error       & rate      & $L_2$ error       & rate      & $L_2$ error       & rate      \\ \hline
                                  & 2                    & 7.27E-02          & --        & 7.36E-02          & --        & 1.64E-01          & --        \\
                                  & 4                    & 4.54E-03          & 4.00      & 4.76E-03          & 3.95      & 3.20E-02          & 2.36      \\
                                  & 8                    & 2.35E-04          & 4.27      & 2.43E-04          & 4.29      & 3.29E-03          & 3.28      \\
$p=3$                             & 16                   & 1.42E-05          & 4.05      & 1.49E-05          & 4.03      & 2.82E-04          & 3.54      \\
                                  & 32                   & 8.84E-07          & 4.00      & 9.46E-07          & 3.98      & 2.19E-05          & 3.69      \\
                                  & 64                   & 5.54E-08          & 4.00      & 6.19E-08          & 3.93      & 1.79E-06          & 3.61      \\
                                  & 128                  & 3.47E-09          & 4.00      & 4.32E-09          & 3.84      & 1.31E-07          & 3.77      \\ \hline
                                  & 2                    & 1.82E-02          & --        & 7.51E-02          & --        & 8.33E-02          & --        \\
                                  & 4                    & 4.70E-04          & 5.27      & 4.47E-03          & 4.07      & 5.37E-03          & 3.96      \\
                                  & 8                    & 1.42E-05          & 5.05      & 1.51E-04          & 4.89      & 1.81E-04          & 4.89      \\
$p=4$                             & 16                   & 3.88E-07          & 5.19      & 4.44E-06          & 5.08      & 5.42E-06          & 5.07      \\
                                  & 32                   & 1.04E-08          & 5.22      & 1.26E-07          & 5.14      & 1.56E-07          & 5.12      \\
                                  & 64                   & 2.97E-10          & 5.13      & 3.79E-09          & 5.06      & 4.74E-09          & 5.04      \\
                                  & 128                  & 8.99E-12          & 5.05      & 1.55E-10          & 4.61      & 1.48E-10          & 5.00      \\ \hline
\end{tabular}
\end{table}
\subsection{Entropy Preservation for the Scalar-valued Case}
\label{sec:burgers-entropy-preservation}

We verify the conclusion of Remark~\ref{rmk:burgers-preservation} by reporting entropy preservation on the level of machine epsilon for a variety of element sizes, polynomial degrees, FR $c$ parameters and node setups in Table~\ref{tab:burgers-entropy-preservation}. For the entropy-preserving case, we use Eq.~\eqref{eq:burgers-temporal-state} for the temporal numerical flux and use the entropy-conserving Eq.~\eqref{eq:burgers-spatial-num-flux} without Lax-Friedrichs upwinding. The initial condition is set to
\begin{equation}
    u_0(x) = 0.2 \sin(\pi (x-\frac{\pi}{10}))
\end{equation}
on a domain of $[0,2]\times[0,2]$, which forms a shock before the final time. The use of a two-point numerical flux requires solving all timeslabs simultaneously. We accelerate the solution by initializing the solution as a loosely-converged entropy-stable solution with decoupled solutions on each timeslab, then proceed to the entropy-conserving coupled solution.

\begin{table}[h!]
\centering
\caption{Energy preservation as defined by Eq.~\eqref{eq:burgers-entropy-preservation} for the Burgers' test case. Each case uses an $N\times N$ $1D+1$ Cartesian grid with polynomial degree $p$.}
\label{tab:burgers-entropy-preservation}
\begin{tabular}{|c|cccc|c|}
\hline
FR $c$                    & $N$ & $p$ & Solution Nodes & Volume Nodes & Energy Preservation \\ \hline
\multirow{5}{*}{$c_{DG}$} & 2 & 3 & GLL            & GLL          & -3.37E-15           \\
                          & 2 & 3 & GLL            & GL           & -3.27E-15           \\
                          & 4 & 2 & GL             & GL           & 4.38E-17            \\
                          & 2 & 8 & GLL            & GL           & 9.12E-16            \\
                          & 4 & 5 & GLL            & GL           & -9.11E-18           \\ \hline
\multirow{5}{*}{$c_{Hu}$} & 2 & 3 & GLL            & GLL          & -1.41E-15           \\
                          & 2 & 3 & GLL            & GL           & -1.60E-14           \\
                          & 4 & 2 & GL             & GL           & 1.50E-15            \\
                          & 2 & 8 & GLL            & GL           & 4.53E-16            \\
                          & 4 & 5 & GLL            & GL           & -1.52E-17          \\ \hline
\end{tabular}
\end{table}

\subsection{Entropy Stability for the Scalar-valued Case}
\label{sec:burgers-entropy-stability}

We present a numerical verification of the entropy stability property of Theorem~\ref{thm:entropystability} in Fig.~\ref{fig:entropy-stability-check}. We re-use the test case of Section~\ref{sec:burgers-entropy-preservation}, which uses an entropy-conserving spatial flux. Rather than a two-point temporal numerical state, we use purely upwind numerical flux as specified in Eq.~\eqref{eq:temporalentropystab_bc}.
For a variety of $p=3$ grids, we confirm that numerical entropy is decreasing. We plot the numerical entropy integrated over face $4$ of each element per the numbering convention of Fig.~\ref{fig:reference_elem}.

\begin{figure}[h!]
    \centering
    \includegraphics[height=4 in]{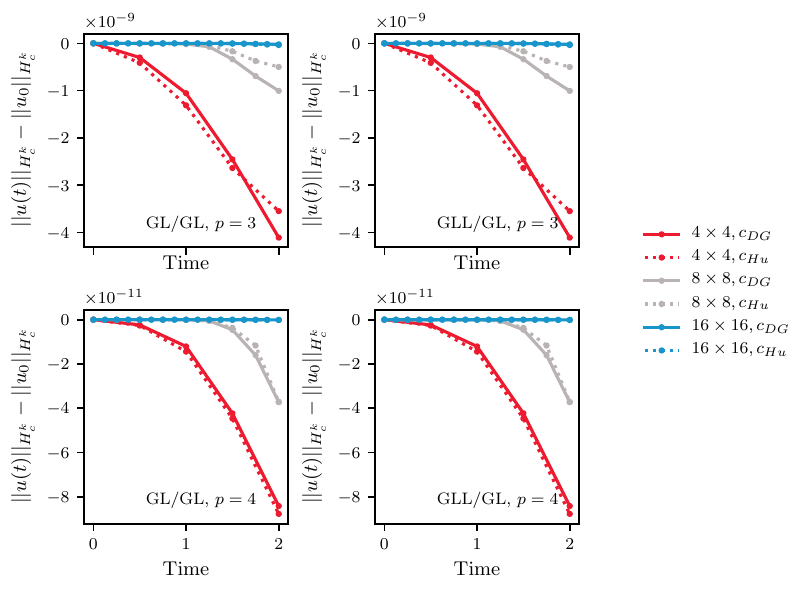}
    \caption{Burgers energy stability as defined by Eq.~\eqref{eq:burgers-entropy-preservation} using $c_{DG}$ and $c_{Hu}$ on grids with $p=3$. In all cases, energy monotonically degreases.}
    \label{fig:entropy-stability-check}
\end{figure}

Furthermore, we demonstrate that numerical entropy change remains negative as $c$ increases in Fig.~\ref{fig:entropy-stability-c-ramp-Burgers}. Alike Fig.~\ref{fig:entropy-stability-check}, we calculate numerical entropy by integrating $H_c^k$ energy over the surface of interest. We thus confirm that the FR portion of energy remains small per Theorem~\ref{Thm:innerproduct_in_Sobolev}; in fact, for the initial condition used here, the scheme is entropy-stable until $c$ is large.

\begin{figure}[h!]
    \centering
    \includegraphics[height = 3 in]{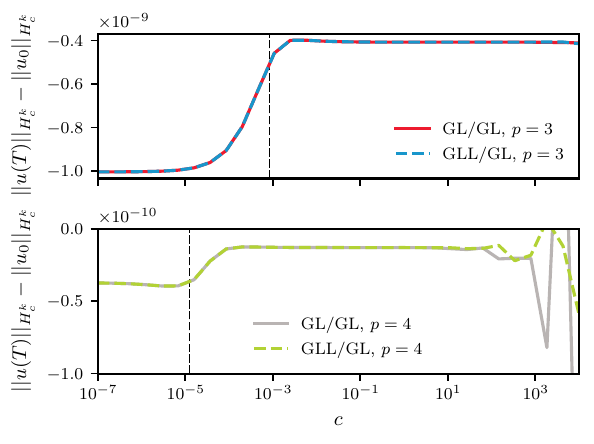}
    \caption{Burgers' $H_c^k$ energy change on an $8\times8$ grid. The first node label is solution and second set is flux (e.g., GL/GLL is GL solution nodes and GLL flux nodes). The vertical dashed line indicates $c_{Hu}$. The lines appear superimposed until $c$ is very large.}
    \label{fig:entropy-stability-c-ramp-Burgers}
\end{figure}

\section{Space-Time Entropy Stable Schemes for General Conservation Laws}
\label{sec:results-Euler}

We proceed to the case where numerical entropy is a nonlinear function without an inner-product form. Such a case arises in vector-valued PDEs such as the Euler or Navier-Stokes systems of fluid dynamics.
\subsection{Numerical Experiments Using the ST-NSFR Scheme}

The vector-valued nonlinear problem uses the $1D+1$ Euler equations, 
\begin{equation}
    \frac{\partial}{\partial t}\begin{bmatrix}
        \rho \\ \rho v \\ E
    \end{bmatrix}
    + \frac{\partial}{\partial x} \begin{bmatrix}
        \rho v \\ \rho v^2 + p \\ v(E + p) \end{bmatrix}
= \bm{0}^T
\end{equation}
where $\rho$ is density, $v$ is velocity, $E$ is total energy, and $p=(\gamma-1) (E - \rho u^2/2)$ is pressure calculated by the ideal gas law with a gas constant $\gamma=1.4$. The numerical entropy function for the Euler equations is $s(\bm{u})=-\rho \log(p \rho^{-\gamma}) / (\gamma - 1)$, with entropy variables 
\begin{equation}
    \bm{v}(\bm{u}) = \begin{bmatrix}
        \gamma + 1 - \log(p \rho^{-\gamma}) - E(\gamma-1)/p \\
        (\gamma-1) \rho v / p \\
        -\gamma(\gamma-1) / p
    \end{bmatrix}.
\end{equation}
The entropy potential is $\phi=\rho$.
For convergence tests, we use the same exact solution as in~\cite{gassner2016split}, which is written in conservative variables as
\begin{equation}
    \bm{u}_{\text{exact}}(x,t) = \begin{bmatrix}
        2 + \frac{1}{10}\sin(\pi (x-2t)) \\
        2 + \frac{1}{10}\sin(\pi (x-2t)) \\
        \left(2 + \frac{1}{10}\sin(\pi (x-2t))\right)^2
    \end{bmatrix},
\end{equation}
with the manufactured source defined to satisfy $\partial \bm{u}_{exact}/\partial t + \partial \bm{f}(\bm{u}_{exact})/\partial x = \bm{q}(x,t)$,
\begin{equation}
        \bm{q}(x,t) = \begin{bmatrix}
        -\frac{ \pi}{10}  \cos (\pi  (x-2 t)) \\
        \frac{ \pi}{100} \cos (\pi  (x-2 t)) \big(5 (7 \gamma -9)+2 (\gamma -1) \sin (\pi  (x-2 t))\big) \\
        \frac{ \pi}{100} \cos (\pi  (x-2 t)) \big(5(7 \gamma -15)+2 (\gamma -2) \sin (\pi  (x-2 t))\big)
    \end{bmatrix}.
\end{equation}
The source is applied explicitly per the preceding equation. 
{Note that we find a different source compared to that published by Gassner~\cite{gassner2016split}, and have verified its correctness.}
The exact solution defines the Dirichlet boundary condition at $t=0$. The domain is $[0,2]\times[0,2]$. 
We use the temporal numerical state function found by Friedrich \textit{et al.}~\cite{friedrich2019entropy} for the Euler equations,
\begin{equation}
    \bm{f}_t(\bm{u}_i,\bm{u}_j) = \begin{bmatrix}
        \rho^{\ln} \\
        \rho^{\ln} \{\!\!\{v \}\!\!\} \\
        \frac{\rho^{\ln}}{2 \beta^{\ln} (\gamma-1)} + \rho^{\ln}
        \tilde{v}
    \end{bmatrix}
\end{equation}
where we use the log-average and arithmetic mean
\begin{equation}
    a^{\ln} = \frac{a_i-a_j}{\ln(a_i) - \ln(a_j)} ,\quad 
    \{\!\!\{a\}\!\!\}=\frac{1}{2}(a_i+a_j).
\end{equation}
We calculate according to the numerically-stable routine of Ismail and Roe~\cite{ismail2009affordable}, and take $\tilde{v} = \{\!\!\{ v \}\!\!\} ^2 - \frac{1}{2}\{\!\!\{ v^2 \}\!\! \}$ and $\beta=v/(2p)$.
The Chandrashekar flux with Ranocha's pressure equilibrium correction is used in the spatial dimension~\cite{chandrashekar2013kinetic,ranocha2018comparison}. 
For convergence tests, we use upwind numerical fluxes per Eq.~\eqref{eq:temporalentropystab_bc} and add matrix dissipation to the spatial numerical flux per~\cite[Appendix A]{gassner2018br1}. We evaluate the interface jump using entropy variables when applying spatial numerical dissipation.
The spatial numerical flux is chosen consistently with the numerical experiments of Friedrich~\textit{et al.}~\cite{friedrich2019entropy}.

We solve the implicit system as described in Section~\ref{sec:lin-adv-results}. $L_2$ errors are presented in Tables~\ref{tab:conv-euler-GLLGL} and~\ref{tab:conv-euler-GLGL} for uncollocated and collocated nodes respectively, reflecting consistent $p+1$ order of accuracy.

\begin{table}[h!]
\centering
\caption{Convergence of the Euler ST-NSFR method using GLL solution nodes and GL flux nodes.}
\label{tab:conv-euler-GLLGL}
\begin{tabular}{|c|c|ll|ll|}
\hline
\multirow{2}{*}{Convergence test} & \multirow{2}{*}{$N$} & \multicolumn{2}{c|}{$c_{DG}$} & \multicolumn{2}{c|}{$c_{Hu}$} \\
                                  &                      & $L_2$ error       & rate      & $L_2$ error       & rate      \\ \hline
                                  & 2                    & 7.40E-02          & --        & 7.57E-02          & --        \\
                                  & 4                    & 5.04E-03          & 3.88      & 5.00E-03          & 3.92      \\
                                  & 8                    & 3.57E-04          & 3.82      & 3.58E-04          & 3.80      \\
$p=3$                             & 16                   & 2.36E-05          & 3.92      & 2.36E-05          & 3.92      \\
                                  & 32                   & 1.50E-06          & 3.97      & 1.51E-06          & 3.97      \\
                                  & 64                   & 9.38E-08          & 4.00      & 9.42E-08          & 4.00      \\ \hline
                                  & 2                    & 1.97E-02          & --        & 1.97E-02          & --        \\
                                  & 4                    & 1.19E-03          & 4.04      & 1.19E-03          & 4.05      \\
                                  & 8                    & 4.52E-05          & 4.72      & 4.52E-05          & 4.72      \\
$p=4$                             & 16                   & 1.53E-06          & 4.89      & 1.53E-06          & 4.89      \\
                                  & 32                   & 4.87E-08          & 4.97      & 4.88E-08          & 4.97      \\
                                  & 64                   & 1.53E-09          & 4.99      & 1.53E-09          & 4.99      \\ \hline
\end{tabular}
\end{table}

\begin{table}[h!]
\centering
\caption{Convergence of the Euler ST-NSFR method using collocated GL solution nodes and flux nodes.}
\label{tab:conv-euler-GLGL}
\begin{tabular}{|c|c|ll|ll|}
\hline
\multirow{2}{*}{Convergence test} & \multirow{2}{*}{$N$} & \multicolumn{2}{c|}{$c_{DG}$} & \multicolumn{2}{c|}{$c_{Hu}$} \\
                                  &                      & $L_2$ error       & rate      & $L_2$ error       & rate      \\ \hline
                                  & 2                    & 8.53E-02          & --        & 8.57E-02          & --        \\
                                  & 4                    & 6.73E-03          & 3.67      & 6.78E-03          & 3.66      \\
                                  & 8                    & 5.10E-04          & 3.72      & 5.27E-04          & 3.69      \\
$p=3$                             & 16                   & 5.86E-05          & 3.12      & 6.13E-05          & 3.10      \\
                                  & 32                   & 4.20E-06          & 3.80      & 5.52E-06          & 3.47      \\
                                  & 64                   & 2.64E-07          & 4.00      & 3.29E-07          & 4.07      \\ \hline
                                  & 2                    & 1.84E-02          & --        & 2.40E-02          & --        \\
                                  & 4                    & 1.17E-03          & 3.97      & 1.34E-03          & 4.16      \\
                                  & 8                    & 3.43E-05          & 5.09      & 5.13E-05          & 4.71      \\
$p=4$                             & 16                   & 1.11E-06          & 4.95      & 1.91E-06          & 4.75      \\
                                  & 32                   & 3.50E-08          & 4.98      & 6.85E-08          & 4.80      \\
                                  & 64                   & 1.10E-09          & 5.00      & 2.30E-09          & 4.89      \\ \hline
\end{tabular}
\end{table}
\subsection{Entropy Preservation for the Vector-valued Case}
\label{sec:Euler-EP-results}

We proceed by confirming in Table~\ref{tab:preservation-Euler} that the scheme is entropy-preserving when the two-point numerical state is also applied as the numerical flux as defined in Eq.~\eqref{eq:temporalentropycons_bc}. This is true on a variety of grid sizes, polynomial degrees, and node choices.  We use a similar discontinuous initial condition as that of~\cite{friedrich2019entropy}, given in primitive variables as
\begin{equation}
    (\rho_0, v_0, p_0) = 
    \begin{cases}
        (1,0,1), \ \ \ 0\leq x \leq 0.3 \\
        (1.125,0,1.1), \ \ \ 0.3 < x \leq 2.
    \end{cases}
    \label{eq:disc-IC}
\end{equation}
The entropy-conserving spatial flux of~\cite{chandrashekar2013kinetic,ranocha2018comparison} is used with no additional dissipation.

\begin{table}[h!]
\centering 
\caption{Entropy preservation per Eq.~\eqref{eq:entropyconservation} for the Euler test case using $c_{DG}$.}
\label{tab:preservation-Euler}
\begin{tabular}{|c|cccc|c|}
\hline
N & p & Solution Nodes & Volume Nodes & Energy Preservation \\ \hline
2 & 3 & GLL            & GLL          & -1.76E-14           \\
2 & 3 & GLL            & GL           & -2.59E-14           \\
4 & 2 & GL             & GL           & 8.97E-15            \\
2 & 8 & GLL            & GL           & -1.03E-14           \\
4 & 5 & GLL            & GL           & -2.74E-15           \\ \hline
\end{tabular}
\end{table}

\subsection{Entropy Stability for the Vector-Valued Case}
\label{sec:Euler-ES-results}
By Theorem~\ref{thm:entropystability}, we expect entropy stability for the $c_{DG}$ case when upwind temporal numerical state is applied per Eq.~\eqref{eq:temporalentropystab_bc}. We re-use the initial condition of Eq.~\eqref{eq:disc-IC}, modifying the temporal numerical state to be upwind per Eq.~\eqref{eq:temporalentropystab_bc}.  We verify that the ST-NSFR scheme results in monotonically decreasing entropy for all tested node choices, FR $c$ parameters, polynomial degrees, and grid levels in Figure~\ref{fig:entropy-stability-check-Eu}. Furthermore, we observe monotonic decrease in entropy for $c_{Hu}$. Theorem~\ref{Thm:innerproduct_in_Sobolev} predicts that entropy stability will be upheld for a range of non-zero $c$ values, which we verify on $8\times8$ grids in Figure~\ref{fig:entropy-stability-c-ramp-Eu}.

\begin{figure}[h!]
    \centering
    \includegraphics[height=4 in]{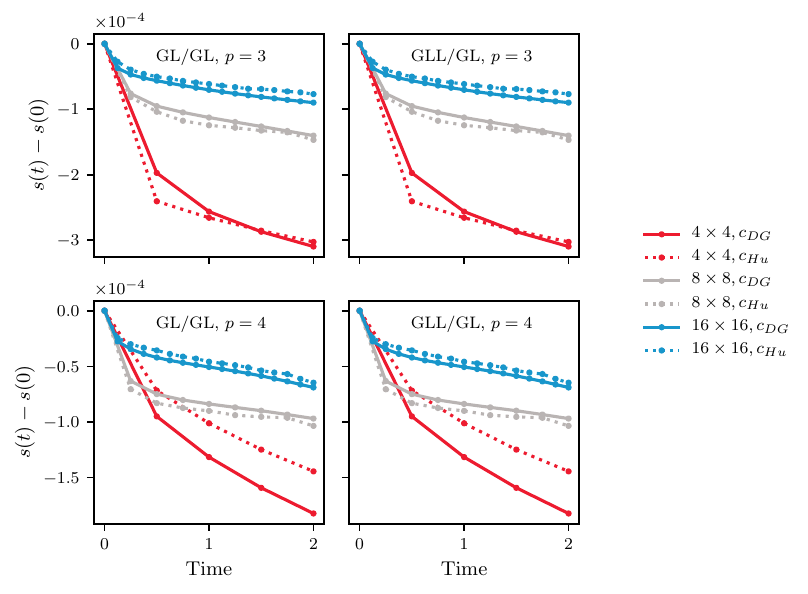}
    \caption{Euler entropy stability as defined by Eq.~\eqref{eq:entropyconservation} using $c_{DG}$ and $c_{Hu}$.}
    \label{fig:entropy-stability-check-Eu}
\end{figure}

\begin{figure}[h!]
    \centering
    \includegraphics[height=2.5in]{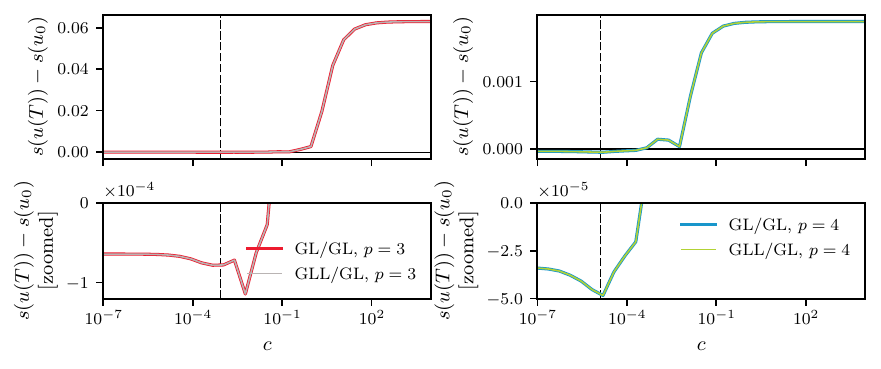}
    \caption{Entropy change over the simulation for the entropy-stable Euler test case. The results for each polynomial are visually indistinguishable. The vertical dashed line indicates $c_{Hu}$.}
    \label{fig:entropy-stability-c-ramp-Eu}
\end{figure}

\section{Convergence Behaviour of Space-time Flux Reconstruction Schemes.}
\label{sec:comp_cost}

Thus far, we have developed a scheme which recovers FR properties and is entropy-stable; in this section, we will demonstrate that it is also computationally efficient to use FR with a nonzero $c$ parameter.
In the method-of-lines approach, FR has long demonstrated a cost advantage over DG by increasing the size of a stable time step size, e.g. by~\cite{huynh_flux_2007,wang2009unifying}.
The ST approach is fully implicit, and therefore, the CFL condition is no longer appropriate.
Therefore, we reuse the test cases from previous sections and compare the number of iterations of the GMRES implicit solution as $c$ is varied.

First, we reuse the test case of Section \ref{sec:lin-adv-results} to demonstrate that the solver cost decreases before the loss of accuracy in Figure~\ref{fig:comp_cost_OOA}. The implicit solver is as described in Section \ref{sec:lin-adv-results}. We keep the solver settings consistent for all cases in this section. The relative cost of the solve is defined as the number of right-hand-side assemblies performed during the implicit solve of the last timeslab, normalized by the setup having GLL solution and GL flux quadrature nodes. We observe in Figure~\ref{fig:comp_cost_OOA} that all four combinations of nodes lead to a decrease in cost before loss of accuracy. Furthermore, collocated GL nodes are able to achieve a nearly identical OOA result at a consistently lower computational cost, highlighting the advantage of flexible node choices.
\begin{figure}[h!]
    \centering
    \includegraphics[height=2.5 in]{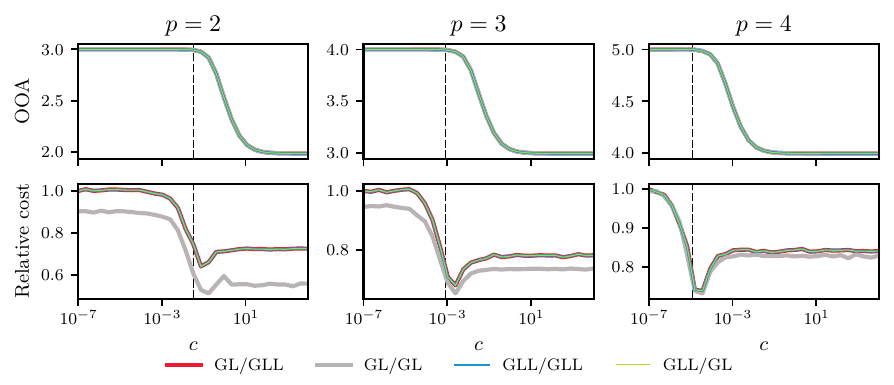}
    \caption{Computational cost for the linear advection case from Section \ref{sec:lin-adv-results} at $32\times32$ elements. With the exception of the $GL/GL$ combination, all results appear superimposed. The vertical dashed line indicates $c_{Hu}$.}
    \label{fig:comp_cost_OOA}
\end{figure}

We proceed to the nonlinear entropy-stable test cases of Sections
\ref{sec:burgers-entropy-stability} and \ref{sec:Euler-ES-results} and compare computational cost in the same way. Figure~\ref{fig:comp_cost_STNSFR} demonstrates a similar decrease in computational cost as the value of $c$ increases despite visually indistinguishable results in the previous sections. When $c$ values remain reasonably small, computational cost consistently decreases to about $70\%$ of the $c_{DG}$ cost.

\begin{figure}[h!]
    \centering
    \includegraphics[height=2.5 in]{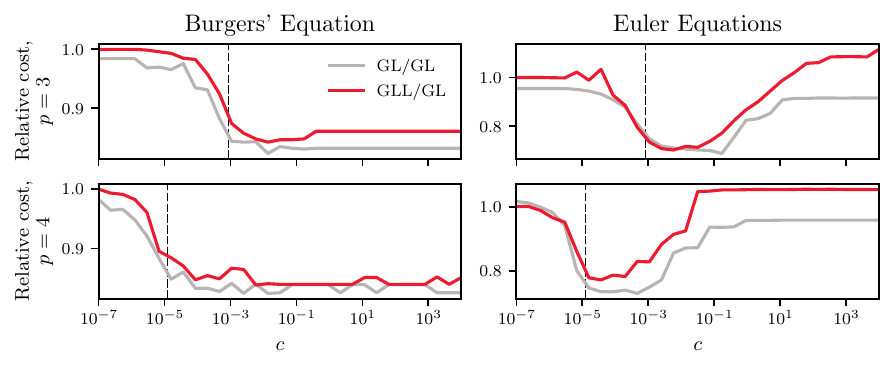}
    \caption{Computational cost for the entropy-stable ST-NSFR cases of Sections
\ref{sec:burgers-entropy-stability} and \ref{sec:Euler-ES-results}. The vertical dashed line indicates $c_{Hu}$.}
    \label{fig:comp_cost_STNSFR}
\end{figure}

\section{Conclusions}

Two space-time FR schemes were formulated with the shared philosophy of using FR in space and DG in time. The ST-ESFR method successfully recovered order-of-accuracy properties of MOL ESFR, supporting the proposed scheme as a space-time equivalent to standard FR implemenations. The main contribution of this work is the development of ST-NSFR, which recovers FR schemes through the choice of FR correction parameter $c$ and allows for flexible node choices. We prove that the novel ST-NSFR scheme is discretely conservative, entropy-preserving and entropy-stable for $c_{DG}$, and entropy-stable for small FR correction parameters $c$. Numerical results using linear advection, Burgers' equation and the Euler equations confirm that both schemes achieve optimal $p+1$ convergence. Both Burgers' and Euler results demonstrated entropy preservation and stability. Furthermore, a consistent reduction in computational cost was observed when using FR compared to DG across all three PDEs.

The theoretical basis for this work could be extended by adapting von Neumann analysis to a space-time framework to elucidate the relationship between the strength of FR and the conditioning of the implicit system. The ST-NSFR method should be tested on cases of industrial interest, such as wall-bounded flows, to investigate the comparison of computational cost between MOL and space-time frameworks.

\section{Data Availability}
The code used to produce the results of this paper, along with all raw results, are available at~\cite{Pethrick2026spacetimeDG}.

\section{Acknowledgements}
We acknowledge the support of the Natural Sciences and Engineering Research Council of Canada (NSERC) Discovery Grant Program [RGPIN-2019-04791] and Canadian Graduate Scholarships - Doctoral program [CGS-D-579552], and the support of McGill University. C.P. recognizes the support of the Zonta International Amelia Earhart Fellowship (2025-2026).


\bibliographystyle{elsarticle-num}
\bibliography{refs.bib}

\appendix

\section{Entropy-stable Numerical State for the Burgers' Equations}
\label{app:burgers_num_state}

We seek an entropy-stable numerical state function for the Burgers' equations, that is, one satisfying~\cite[Eq. (2.32)]{friedrich2019entropy}, 
\begin{equation}
    [[v]] f_t (u_i,u_j) = [[\phi]].
\end{equation}
We have entropy variable $v=u$ and entropy potential $\phi=1/2 \ u^2$ for the Burgers' equations. Thus, we have
\begin{equation}
    (u_i-u_j) f_t(u_i,u_j) = \frac{1}{2}(u_i^2-u_j^2),
\end{equation}
resulting in
\begin{equation}
    f_t (u_i,u_j) = \frac{1}{2}(u_i+u_j).
\end{equation}

\end{document}